\documentclass{article}
\usepackage[utf8]{inputenc}
\usepackage{amsfonts}
\usepackage{amsmath}
\usepackage{amssymb}
\usepackage{amsthm}
\usepackage{xcolor}
\usepackage{bm}
\usepackage{mathtools}
\usepackage{enumitem}
\usepackage{verbatim}
\usepackage{pinlabel}
\usepackage{caption}
\usepackage{subcaption}
\usepackage{tikz}
\usepackage{graphicx}
\usepackage{hyperref}
\usepackage{tikz-cd}
\usepackage{float}
\usepackage{geometry}
\usepackage{soul}

\usepackage[rightcaption]{sidecap}

\DeclarePairedDelimiter\floor{\lfloor}{\rfloor}

\newtheorem{theorem}{Theorem}[section]
\newtheorem{lemma}[theorem]{Lemma}
\newtheorem{proposition}[theorem]{Proposition}

\newtheorem{corollary}[theorem]{Corollary}
\newtheorem{question}[theorem]{Question}

\theoremstyle{definition}
\newtheorem*{example}{Example}

\newtheorem*{remark}{Remark}

\DeclareMathOperator{\interior}{int}

\DeclareMathOperator{\Homeo}{Homeo}

\DeclareMathOperator{\Diffeo}{Diff}

\DeclareMathOperator{\sign}{sign}

\DeclareMathOperator{\ann}{Ann}

\DeclareMathOperator{\supp}{supp}
\DeclareMathOperator{\Star}{star}
\DeclareMathOperator{\coo}{c_{00}}

\newcommand{\fine}{\mathcal{C}^\dagger}
\newcommand{\finearc}{\mathcal{A}^\dagger}
\newcommand{\fiber}{\fine_{\alpha}}
\newcommand{\surf}{S_{g,b}}
\newcommand{\ourfine}{\mathcal{C}_{\mathcal{A}}^\dagger}

\newcommand{\p}[1]{\medskip\noindent\textbf{#1}\textbf{.}}

\newcommand{\pit}[1]{\medskip\noindent\textit{#1}\textit{.}}

\newcommand{\G}{\mathcal{G}}

\newcommand{\dfiber}{d_{\mathcal{C}^{\dagger}_{\alpha}}}

\title{Large flats in large subgraphs of fine curve graphs}
\author{Ryan Dickmann and Roberta Shapiro}
\date{\today}

\begin{document}

\maketitle

\begin{abstract}
    The fine curve graph of a surface is a graph whose vertices are essential simple closed curves and whose edges connect disjoint curves. Following a rich history of hyperbolicity of various graphs associated to surfaces, the fine curve graph was shown to be hyperbolic by Bowden–Hensel–Webb, while the curve graph, obtained from the fine curve graph by collapsing subgraphs corresponding to isotopy classes, was first proven to be hyperbolic by Masur--Minsky.
    
    We show that certain large subgraphs of fine curve graphs, including fibers over a vertex of the curve graph, are not hyperbolic. Indeed, such graphs contain flats of every finite dimension.

    We then compute bounds on distances in fibers over a vertex of the curve graph, which we call single-isotopy-class fine curve graphs.
    
\end{abstract}

\section{Introduction}

The \textit{fine curve graph} of an orientable surface $S_{g,b}$ with $g$ genus and $b$ boundary components, denoted $\fine(S_{g,b}),$ is the graph whose vertices are essential nonperipheral simple closed curves in $S_{g,b}$ and edges connect pairs of disjoint curves. The fine curve graph was originally introduced by Bowden--Hensel--Webb (with the restriction that vertices are smooth curves) to study $\Diffeo_0(S_{g,0})$ \cite{Bowden_Hensel_Webb_2021}. In particular, they showed that its space of quasimorphisms of $\Diffeo_0(S_{g,0})$ is infinite-dimensional, and, as a corollary, that $\Diffeo_0(S_{g,0})$ is not uniformly perfect. These results also apply to $\Homeo_0(S_{g,0})$ and the fine curve graph as we defined it above.

The fine curve graph differs from the usual curve graph, $\mathcal{C}(\surf)$, which has as its vertices \emph{isotopy classes of} curves. In the cases that $(g,b)=(1,0),\ (1,1), \text{ or } (0,4),$ we have that $\fine(\surf)$ is disconnected since any essential simple closed curves in distinct isotopy classes must intersect. When $g=0$ and $b<4$ there are no essential curves, so the fine curve graph is empty. Therefore, we will not consider these cases going forward.

Let $\mathcal{A}$ be a set of isotopy classes of essential nonperipheral simple closed curves in $\surf.$ Define $\fine_{\mathcal{A}}(\surf)$ to be the subgraph of $\fine(\surf)$ induced by all curves in the isotopy classes in $\mathcal{{A}}.$ When $\mathcal{A}$ consists of a single isotopy class $[\alpha]$, we instead denote the corresponding subgraph $\fiber(\surf)$. We will refer to $\fiber(\surf)$ as a \emph{single-isotopy-class fine curve graph}. By a similar proof to Long--Margalit--Pham--Verberne--Yao \cite[Section 3]{LMPVY}, we have that $\fine_{\mathcal{A}}(\surf)$ is connected if and only if the subgraph of $\mathcal{C}(\surf)$ induced by $\mathcal{A}$ is connected. We will assume henceforth that $\mathcal{A}$ is chosen such that it induces a connected subgraph of $\mathcal{C}(\surf)$.

Bowden--Hensel--Webb showed that $\fine(S_{g,0})$ is infinite diameter \cite{Bowden_Hensel_Webb_2021}.
In fact, there are isotopic curves that are arbitrarily far apart, so a similar argument shows that $\fine_{\mathcal{A}}(\surf)$ is infinite diameter. Additionally, Bowden--Hensel--Webb prove that $\fine(S_{g,0})$ is Gromov hyperbolic \cite{Bowden_Hensel_Webb_2021}. Contrary to this, using ideas related to witness subsurfaces and subsurface projections, we show that many $\fine_{\mathcal{A}}(\surf)$ are \emph{not} Gromov hyperbolic. 

Let $\coo$ be the space of real sequences that are eventually 0 endowed with the metric induced by the $\ell^\infty$ (supremum) norm. For a graph $G,$ define the \emph{star} of a vertex $v$ of $G,$ denoted $\Star(v)$, to be $v$ along with all of its neighbors.

\begin{theorem}\label{thm:linfty}
        Let $\surf$ be an orientable surface with $g\geq 1$ or $b\geq 4.$ Let $\beta$ be an essential nonperipheral simple closed curve in $\surf$ and let $\mathcal{B}$ be the set of isotopy classes corresponding to vertices of $\mathcal{C}(\surf)-\Star{\beta}$. Let $\mathcal{A}\subseteq \mathcal{B}$ be a set of isotopy classes such that the subgraph of $\mathcal{C}(\surf)$ induced by $\mathcal{A}$ is connected.
 Then there is a quasi-isometric embedding of $\coo$ into $\mathcal{C}_{\mathcal{A}}^{\dagger}(\surf)$.
 \end{theorem}

 The following is an immediate consequence of Theorem~\ref{thm:linfty}.

\begin{corollary}\label{maincorollary}
    Let $\surf$ and $\mathcal{A}$ be as in the statement of Theorem~\ref{thm:linfty}. Then $\mathcal{C}_{\mathcal{A}}^{\dagger}(\surf)$ contains a quasi-isometrically embedded copy of $\mathbb{Z}^N$ for all positive integers $N.$ Thus $\fine_{\mathcal{A}}(\surf)$ is not Gromov hyperbolic. In particular, the subgraph of $\fine(\surf)$ induced by a single isotopy class is not hyperbolic. 
\end{corollary}

Corollary~\ref{maincorollary} implies that $\mathcal{C}_{\mathcal{A}}^{\dagger}(\surf)$ cannot quasi-isometrically embed into $\fine(\surf)$; in particular, the inclusion is not a quasi-isometric embedding. We further have that although $\fine(S_{g,b})$ is hyperbolic, it can be expressed as a (non-disjoint) union of two connected non-hyperbolic spaces.

\begin{corollary}\label{cor:coveredbynonhyp}
    Let $\surf$ be an orientable surface with $g\geq 1$ or $b\geq 4.$ Then $\fine(\surf)$ can be expressed as the (non-disjoint) union of two connected non-hyperbolic spaces.
\end{corollary}

We also study the geometry of the single-isotopy-class fine curve graph. Our work is an adaptation of Lemma 4.5 of Bowden--Hensel--Mann--Militon--Webb \cite{BHMMW} which gives an upper bound for distances in the fine curve graph of the torus (see Remark~\ref{rmk:torus} for more information). We define a similar notion for the crossing number (see Section \ref{sec:four})  between isotopic curves and use it to provide distance bounds on curves in the single-isotopy-class fine curve graph. This gives us the following description of the coarse geometry of the single-isotopy-class fine curve graph.

\begin{theorem} \label{thm:crossing}
    The distance between curves in the single-isotopy-class fine curve graph is bounded above and below by linear functions of the crossing numbers between the two curves.
\end{theorem}

These bounds can be found in Propositions~\ref{prop:distboundssep} (for separating isotopy classes), \ref{prop:distboundssingleisotopyclass} (for nonseparating isotopy classes), and \ref{prop:disttorus} (for curves on tori).

From our analysis, we can also explicitly produce a sequence of curves $\{\alpha_n\}_{n=0}^\infty$ in each case such that $\dfiber(\alpha_0, \alpha_n) = n$ (Corollary \ref{cor:exactdistance}).

\p{Curve graphs, past and present}  The \textit{curve graph} of a surface is a classically studied object that has as its vertices isotopy classes of essential simple closed curves while edges connect isotopy classes that admit disjoint representatives. We obtain the fine curve graph from the curve graph by collapsing each subgraph corresponding to a single isotopy class of curves to a single vertex.

Curve graphs have been used to study the mapping class group (group of boundary-preserving homeomorphisms up to isotopy) and to find out more about surfaces in general. We are often interested in more geometric aspects of curve graphs, such as hyperbolicity (Masur--Minsky \cite{MM}), or combinatorial, such as the automorphism group (Ivanov \cite{Ivanov}). Moreover, variants of curve graphs, such as arc graphs \cite{Hatcher}, pants graphs \cite{HT}, and $k$-curve graphs \cite{kcurve}, have also been constructed and studied.

Some direct analogues of these questions have been explored for fine curve graphs. For example, hyperbolicity is proven in \cite{Bowden_Hensel_Webb_2021}; automorphism groups are found in \cite{LMPVY}, \cite{kk}, and \cite{swwz}; and several variants are introduced in Le Roux--Wolff \cite{LRW}, Booth--Minahan--Shapiro \cite{BMS}, Booth \cite{booth1} \cite{booth2}, and Shapiro \cite{shapiro}.

Since mapping class groups and curve graphs of surfaces are so well-studied, we ask: what are the aspects of fine curve graphs that do \emph{not} have curve graph analogues? One way to approach such a question is to see what information becomes hidden when we pass from the fine curve graph (which we know a bit about) to the curve graph (which we know quite a lot about). 

In fact, the following result on the topology of the fine curve graph is known. Consider fine curve graphs and curve graphs as simplicial complexes by setting $k$-simplices as $k+1$-cliques in the graphs.
Let $f:\fine(\surf)\to \mathcal{C}(\surf)$ be the collapsing map sending curves to their isotopy classes; Dickmann--Himes--Nolte--Shapiro show that $f$ is a homotopy equivalence \cite{dhns}. In fact, they show that single-isotopy-class fine curve graphs (turned simplicial complex by considering $k+1$-cliques as $k$-simplices) are contractible (as are preimages under $f$ of any simplex in $\mathcal{C}(S)$).

We propose the following questions for further study of the fine curve graphs.

\begin{question} (Asked by Alex Wright)
    Does there exist any quasi-isometric embedding of $\ell^\infty$ into some $\mathcal{C}_{\mathcal{A}}^{\dagger}(\surf)$?
\end{question}

\begin{question}
Classify quasi-isometric embeddings between $\mathcal{C}_{\mathcal{A}}^{\dagger}(\surf)$. In particular, is $\mathcal{C}_{\alpha}^{\dagger}(\surf)$ quasi-isometric to $\mathcal{C}_{\beta}^{\dagger}(\surf)$ when $\alpha$ and $\beta$ are not the same topological type?
\end{question}

\p{Proof outline for Theorem~\ref{thm:linfty}} Our goal in Theorem~\ref{thm:linfty} is to show that there is a quasi-isometric embedding of $\coo$ into $\fine_{\mathcal{A}}(S_{g,b}).$ The main idea of the proof is that we find specially chosen annuli in $S_{g,b}$, called witnesses, and then perform very ``mixing" point pushes on each witness. 

\pit{Witnesses of $\mathcal{C}_{\mathcal{A}}^{\dagger}(\surf)$} Let $\beta$ be as in the statement of Theorem~\ref{thm:linfty}. Every vertex in $\fine_{\mathcal{A}}(\surf)$ intersects $\beta$ essentially by construction. We may then take infinitely many disjoint (except potentially at the boundary) annuli whose core curves are isotopic to $\beta.$ Each vertex of $\fine_{\mathcal{A}}(\surf)$ intersects every one of the annuli essentially since none of the curves can be isotoped to be disjoint from any of the annuli, and we refer to the annuli as witnesses for the curve.

Using results inspired by Bowden--Hensel--Webb, in Section \ref{sec:two} we show these point push maps act with linear speed on a specifically chosen graph and have unbounded orbits (akin to the action of pseudo-Anosov homeomorphisms on curve graphs). We then use this in Section \ref{sec:prooftheoremlinfty} to show the image of some $\alpha$ whose isotopy class is in $\mathcal{A}$ under iteration of these point pushes becomes a quasi-isometrically embedded copy of $\coo$ and then construct explicit paths in $\mathcal{C}_{\mathcal{A}}^{\dagger}(\surf)$.

\p{Proof outline for Theorem~\ref{thm:crossing}} In Section \ref{sec:four} we first define the crossing number and discuss some preliminaries on lifts and covers. 

\pit{Crossing number} For two isotopic curves $\beta, \gamma$ in a surface, we define the crossing number $C_{\beta}(\gamma)$ by lifting both curves to the cyclic cover of the surface corresponding to the homotopy class of the given curves. We then count the number of intersections of certain elevations (a connected component of the union of all lifts) to define the crossing number.

We discuss the symmetry of the crossing number in Section \ref{sec:five}, and then use these ideas in Section \ref{sec:six} to prove Theorem~\ref{thm:crossing}.  Our proofs throughout these sections are similar and involve a direct analysis of lifts in the cyclic and universal cover to count the crossing number.

\p{Acknowledgements} The authors thank Dan Margalit for his support and for many conversations. The authors also thank Sarah Koch for extensive conversations and Alex Wright for proposing the question on quasi-isometric embeddings of $\ell^\infty.$ The second author was partially supported by the National Science Foundation under Grant No. DMS-2203431.

\section{Linear-speed action on the fine arc graph of a witness}\label{sec:two}

 Define the \textit{fine arc graph of $\surf$} to be the graph whose vertices are properly embedded (simple) essential arcs in $\surf$ and whose edges connect arcs that are disjoint. In this section, we mostly consider a surface $\Sigma$, an orientable, unpunctured planar surface with at least 2 boundary components. We will denote all planar surfaces by $\Sigma$ and all surfaces that need not be planar by $\surf.$

The main goal of this section is to prove Proposition~\ref{prop:hypactiononfinearcgraph}, which states that there are homeomorphisms isotopic to the identity whose actions on $\mathcal{A}^\dagger(\Sigma)$ are loxodromic-like. That is, we show that there exists a homeomorphism $\varphi:\Sigma\to \Sigma$ such that there exist $K\in \mathbb{R}_{> 0}$ and $L\in \mathbb{R}_{\geq 0}$ such that $\frac{1}{K}n-L\leq d_{\finearc(\Sigma)}(\varphi^n(\gamma),\gamma)$ for all $n$ and for any arc $\gamma.$ We say that such elements act with \textit{linear speed} (rather than being loxodromic) since we do not know whether the graphs in question are hyperbolic. These homeomorphisms will allow us to find a quasi-isometrically embedded copy of $\coo$ in $\fine_{\mathcal{A}}(S_{g,b}),$ where $S_{g,b}$ is the original surface.

Along the way, we adapt results of Bowden--Hensel--Webb \cite{Bowden_Hensel_Webb_2021} about bounds on distances in the fine arc graph in terms of those in the surviving arc graph (a subgraph of the arc graph that we define later).

\begin{proposition}\label{prop:hypactiononfinearcgraph}
    Let $\Sigma$ be a planar surface with at least 2 boundary components and $P\subset \Sigma\setminus\partial\Sigma$ such that $3\leq|P|< \infty.$ Then there is a homeomorphism $\varphi\in\Homeo(\Sigma,P)$ isotopic to the identity in $\Sigma$ that acts with linear speed on $\mathcal{A}^\dagger(\Sigma).$ That is, there are constants $K>0,L\geq 0$ such that, for any arc $\gamma$ in $\Sigma,$ 
    \[\frac{1}{K}|n|-L\leq d_{\finearc(\Sigma)}(\varphi^n(\gamma),\gamma).\]
\end{proposition}

We begin by stating and proving the well-known result that there are point-pushing pseudo-Anosov maps that act loxodromically on arc graphs.

\begin{lemma}\label{cor:pointpushactshyponarcgraphrelp}
    Let $\Sigma$ be a planar surface with at least 2 boundary components. Let $P\subset \Sigma\setminus\partial\Sigma$ be a finite set with $|P|\geq 3.$ Then there exists a point-pushing mapping class of $\Sigma\setminus P$ that acts loxodromically on $\mathcal{A}(\Sigma\setminus P).$
\end{lemma}

\begin{proof}
    By work of Disarlo, $\mathcal{A}(\Sigma\setminus P)$ is quasi-isometric to $\mathcal{C}(\Sigma\setminus P)$ via surgery arguments \cite[Theorem 4.1]{Disarlo}. This quasi-isometry is coarsely equivariant in the sense that a mapping class acting on $\mathcal{A}(\Sigma\setminus P)$ acts in a very similar way on $\mathcal{C}(\Sigma\setminus P).$ (This similarity can be made more rigorous by factoring through the arc and curve graph, which we do not define here.) By work of Kra \cite{Kra}, a point-push along a filling curve in $\Sigma\setminus P$ is a pseudo-Anosov. We conclude that this point-push acts loxodromically on $\mathcal{C}(\Sigma\setminus P)$, and, using the coarse equivariance, therefore on $\mathcal{A}(\Sigma\setminus P)$ as well.
\end{proof}

For a connected, orientable surface $\surf,$ define the \textit{surviving arc graph of $\surf\setminus P$}, denoted $\mathcal{A}^s(\surf \setminus P)$, to be the subgraph of $\mathcal{A}(\surf\setminus P)$ spanned by arcs that remain proper and essential after filling in $P$ (and are therefore proper and essential in $\surf$). For a homeomorphism $\varphi \in \Homeo(S)$ denote by $[\varphi]_{S\setminus P}$ the mapping class of $\varphi$ restricted to $S\setminus P.$

\begin{lemma}\label{lemma:distanceinequalitysurvivingfine}
    Let $S=\surf$ be a connected, orientable surface with $b\geq 2$, and $P\subset\interior(S)$ be a finite set. Let $\varphi\in \Homeo(S)$ such that $\varphi(P)=P,$ 
    then for any vertex $\gamma$ of $\finearc(S)$ disjoint from $P,$ we have that
    \[d_{\mathcal{A}^s(S\setminus P)}\left([\gamma]_{S\setminus P}, [\varphi]^i_{S\setminus P}[\gamma]\right) \leq d_{\finearc(S)}(\gamma,\varphi^i(\gamma))\]
    for all $i\in \mathbb{Z}.$
\end{lemma}

The proof of Lemma~\ref{lemma:distanceinequalitysurvivingfine} is analogous to that of Lemma 4.2 of Bowden--Hensel--Webb, but we replace all instances of ``curve" with ``arc" \cite{Bowden_Hensel_Webb_2021}. It is further important to note that $\finearc(S)$ is connected for surfaces with boundary but not punctures (the proof is analogous to that of the connectivity of fine curve graphs).

\begin{proof}[Proof of Proposition~\ref{prop:hypactiononfinearcgraph}]
    By Lemma~\ref{cor:pointpushactshyponarcgraphrelp}, there is a point-pushing homeomorphism $\varphi\in \Homeo_0(\Sigma)$ with $\varphi(P)=P$ such that the mapping class $[\varphi]$ acts loxodromically on $\mathcal{A}(\Sigma\setminus P).$ Let $K$ and $L'$ be positive integers such that $\frac{1}{K}|n| - L' \leq d_{\mathcal{A}(\Sigma\setminus P)}([\varphi]^n(\alpha),\alpha).$ 

    Let $\gamma$ be an arc disjoint from $P$ that is a vertex of $\mathcal{A}^{\dagger}(\Sigma)$. We then have the following sequence of inequalities:
    \begin{align*}
        \frac{1}{K}|n| - L' &\leq d_{\mathcal{A}(\Sigma\setminus P)}([\varphi]_{\Sigma\setminus P}^n[\gamma]_{\Sigma\setminus P},[\gamma]_{\Sigma\setminus P})\\
        &\leq d_{\mathcal{A}^s(\Sigma \setminus P)}([\varphi]_{\Sigma\setminus P}^n[\gamma]_{\Sigma\setminus P},[\gamma]_{\Sigma\setminus P})\\
        &\leq d_{\finearc(\Sigma)}(\gamma,\varphi^n(\gamma)).
    \end{align*}

    It remains to show the above for $\gamma$ with $\gamma\cap P\neq \emptyset.$ Let $\gamma'$ be a perturbation of $\gamma$ with $\gamma'\cap \gamma = \emptyset$ and $\gamma'\cap P = \emptyset.$ It follows that $\varphi^n(\gamma)\cap \varphi^n(\gamma')=\emptyset$ for all $n.$ We then have that \[d_{\finearc(\Sigma)}(\gamma',\varphi^n(\gamma'))-2\leq d_{\finearc(\Sigma)}(\gamma,\varphi^n(\gamma)) \leq d_{\finearc(\Sigma)}(\gamma',\varphi^n(\gamma'))+2.\] Taking $L=L'+2,$ we obtain 
    \[
    \frac{1}{K}|n| - L \leq d_{\finearc(\Sigma)}(\gamma,\varphi^n(\gamma))
    \]
    for all vertices $\gamma$ in $\finearc(\Sigma)$ and $n\in \mathbb{Z}.$
\end{proof}

The results in this section are significant only when the arc graphs are connected (and indeed they are connected by work of Long--Margalit--Pham--Verberne--Yao \cite{LMPVY}). Connectivity of surviving arc graphs follows by considering arcs in $\Sigma$ to be (isotopy classes of) arcs in $\Sigma \setminus P,$ potentially after some deformation to avoid $P$. 

If $\finearc(\Sigma)$ is hyperbolic, Proposition~\ref{prop:hypactiononfinearcgraph} implies that $\varphi$ acts loxodromically on $\finearc(\Sigma).$ However, that is not strictly necessary for the proof of our theorem. 

\section{A quasi-isometric embedding of $\coo$}\label{sec:prooftheoremlinfty}

The goal of this section is to prove Theorem~\ref{thm:linfty}. In Section~\ref{subsec:constructinghomeo}, we (somewhat) explicitly construct homeomorphisms of $\surf$ that allow for the embedding. In Section~\ref{subsec:prooftheoremlinfty}, we complete the proof of Theorem~\ref{thm:linfty}.

\subsection{Constructing homeomorphisms of $\surf$ and paths in $\mathcal{C}_{\mathcal{A}}^\dagger(\surf)$}\label{subsec:constructinghomeo}

In this section, we construct homeomorphisms that will allow us to embed $\coo$ into $\ourfine(\surf).$ We will then use our knowledge of these homeomorphisms to build paths (and therefore bound distances) in $\ourfine(\surf).$

\p{Summary} To build homeomorphisms, we will begin with countably many disjoint annuli, one for each entry of an element of $\coo.$ These annuli will ``witness'' all curves in the collection $\mathcal{A}$, so each curve is impacted by any homeomorphism supported on any of these annuli.

Then we will use a model homeomorphism from Section~\ref{sec:two} that acts with linear speed on the fine arc graph of an annulus to build homeomorphisms supported on each of the annuli from the previous paragraph. 

The quasi-isometric embedding from $\coo$ to $\ourfine(\surf)$ will then be defined by approximating an element of $\coo$ as an integer element of $\coo$ and sending this integer element to the image of a fixed curve under a product of homeomorphisms. The integer in the $i$th slot of the element of $\coo$ becomes the power of a homeomorphism supported on the $i$th annulus. 

The lower bound for the distance between the fixed curve and its image is given by the lower bound constants in Section~\ref{sec:two}. The upper bound is given by the construction of explicit paths in this section.

\p{Set-up and a result about an auxiliary graph} Let $A=S^1\times I$ be an annulus. A \emph{multiarc} in $A$ is a finite collection of disjoint (including at the endpoints) proper simple essential arcs in $A.$ Let $\theta_1,\ldots,\theta_\ell\subset S^1$ be a collection of distinct points in the circle. We can view these points as subsets of each boundary component of $A.$ Similarly, let $\sigma_1,\ldots,\sigma_\ell \subset S^1$ be a collection of distinct points in $S^1$ such that $\sigma_i\neq \theta_j$ for all $i,j.$ 

Define a \emph{$\theta$-multiarc} to be a multiarc that consists of $\ell$ arcs, where the $i$th arc connects $\{\theta_i\}\times\{0\}$ to $\{\theta_i\}\times \{1\}.$ Similarly, define a \emph{$\sigma$-multiarc} to be a multiarc that consists of $\ell$ arcs, where the $i$th arc connects $\{\sigma_i\}\times\{0\}$ to $\{\sigma_i\}\times \{1\}.$ The \emph{$i$th component} of a $\theta$-multiarc (resp. $\sigma$-multiarc) is the component that connects $\{\theta_i\}\times \{0\}$ to $\{\theta_i\}\times\{1\}$ (resp. $\{\sigma_i\}\times \{0\}$ to $\{\sigma_i\}\times\{1\}$). The $i$th components of a $\theta$-multiarc and a $\sigma$-multiarc are called \emph{corresponding components.}

We now define a graph $\G.$ The vertices of $\G$ are all $\theta$-multiarcs and $\sigma$-multiarcs. An edge connects two multiarcs if they are disjoint (including at their endpoints). (We remark that $\G$ is naturally bipartite, with one part of the partition consisting of $\theta$-multiarcs and the other part consisting of $\sigma$-multiarcs, though that is not important to our argument.)

We now have the following connectivity result about $\G.$

\begin{lemma}\label{lemma:connectivityofextragraph}
    Let $\G$ be a multiarc graph of an annulus $A$ as defined above. Let $\alpha,\beta$ be two $\theta$-multiarcs such that there is a homotopy of $A$ (relative to the boundary) taking $\alpha$ to $\beta.$ Then $\alpha$ and $\beta$ are in the same connected component of $\G.$
\end{lemma}

The proof of this lemma is mainly inspired by the proof of the connectivity of the fine arc graph \cite{LMPVY}, which is itself based on the outline of the proof of connectivity of the fine curve graph \cite{Bowden_Hensel_Webb_2021}. We include all details for completeness and we will clarify certain distinctions after we conclude the proof.

\begin{proof}[Proof of Lemma~\ref{lemma:connectivityofextragraph}]
    We will use the fact that $\alpha$ and $\beta$ are homotopic to produce a path from $\alpha$ to $\beta$ in $\G$ that alternates between $\theta$-multiarcs and $\sigma$-multiarcs.

    Since $\alpha$ and $\beta$ are homotopic, there exists an isotopy \[H:A\times I\to A\] fixing $\partial A$ such that $H(x,0)=x$ for all $x\in A$ and $H(\alpha,1)=\beta.$ 

    We will now define an open cover of $I.$ For all $t\in I,$ define
    \begin{align*}
        U_t =\ &\text{the connected component of } \{s\in I\ |\ \text{the } i\text{th component of } H(\alpha, s) \text{ intersects the }\\ & \qquad\qquad\qquad\qquad\qquad\qquad\qquad\qquad \ \  j\text{th component of } H(\alpha, t) \text{ only if } i=j\}\\ &\text{containing }t.
    \end{align*}
    By the compactness of $I,$ and since $\{U_t\}_{t \in I}$ is an open cover of $I$, there exists a finite subcover $\{U_{t_1},\ldots,U_{t_k}\}$ of $I.$ Without loss of generality, by excluding some of these sets, shrinking some of the sets, including at most 2 other sets, or reordering the sets, we have that:
    \begin{itemize}
        \item $U_h\cap U_i\cap U_j=\emptyset$ for all distinct $h,i,j,$
        \item $0\in U_{t_i}$ if and only if $t_i=0$ and $1\in U_{t_i}$ if and only if $t_i=t_k=1,$ and
       
        \item there exists a natural ordering on the $U_{t_i}$'s such that $t_i<t_j$ if and only if the left endpoint of $U_{t_i}$ is less than the left endpoint of $U_{t_j}$.
        
    \end{itemize}
    Let $s_i\in U_{t_i}\cap U_{t_{i+1}}$ for $1\leq i\leq k-1.$ We note that $H(\alpha,s_i)$ and $H(\alpha,t_i)$ (and similarly, $H(\alpha, s_i)$ and $H(\alpha,t_{i+1})$) are each $\theta$-multiarcs such that only corresponding components of arcs (potentially) intersect. This implies that $\{\sigma_j\}\times \{0\}$ is in the same connected component of $A\setminus\left(H(\alpha,t_i)\cup H(\alpha,s_i)\right)$ as $\{\sigma_j\}\times\{1\}$ for all $j.$

    For this reason, for $1\leq i\leq k-1$, we can define a $\sigma$-multiarc 
    $\alpha_{q_{2i-1}}$ disjoint from $H(\alpha,t_i)$ and $H(\alpha,s_i)$ and another $\sigma$-multiarc $\alpha_{q_{2i}}$ disjoint from $H(\alpha,s_i)$ and $H(\alpha,t_{i+1}).$ We therefore have the following path in $\G:$
    \[\alpha=H(\alpha,t_1)-\alpha_{q_1}-H(\alpha,s_1)-\alpha_{q_2}-\cdots-\alpha_{q_{2k-2}}-H(\alpha,t_k)=\beta.\]
    We point out that the path alternates between $\theta$-multiarcs and $\sigma$-multiarcs. We conclude that $\alpha$ and $\beta$ are in the same connected component of $\G.$
\end{proof}

The proof of Lemma~\ref{lemma:connectivityofextragraph} closely follows that of the connectivity of the arc graph, but with the following two distinctions. First, we work with multiarcs instead of single arcs, causing a mild difference in our definition of the open cover. Second, we prescribe specific possible endpoints for our multiarcs.
As a result, $\G$ is disconnected, with connected components in bijection with $\mathbb{Z}$ (corresponding to Dehn twists about a boundary component of the annulus).

\p{The construction of homeomorphisms} 
Take $\beta$ and $\mathcal{A}$ to be as in the statement of Theorem~\ref{thm:linfty}. Let $\alpha\in[\alpha]\in\mathcal{A}$. Let $A\subset \surf$ be an annulus with a curve isotopic to $\beta$ as the core curve. To make it explicit, we impart the following coordinates on $A$: $A=S^1\times I=I/(\{0\}\sim \{1\}) \times I.$ Without loss of generality, $\alpha \cap A = \bigcup_J\left(\{\theta_j\}\times I\right)$ for $J$ a finite set. Define $A_i\subset A$ for $i$ positive to be 
\[A_i=S^1\times \left[\frac{1}{2^i },\frac{1}{2^{i-1}}\right].\]

\begin{figure}[h]
\begin{center}
\begin{tikzpicture}
    \node[anchor = south west, inner sep = 0] at (0,0) {\includegraphics[width=2.5in]{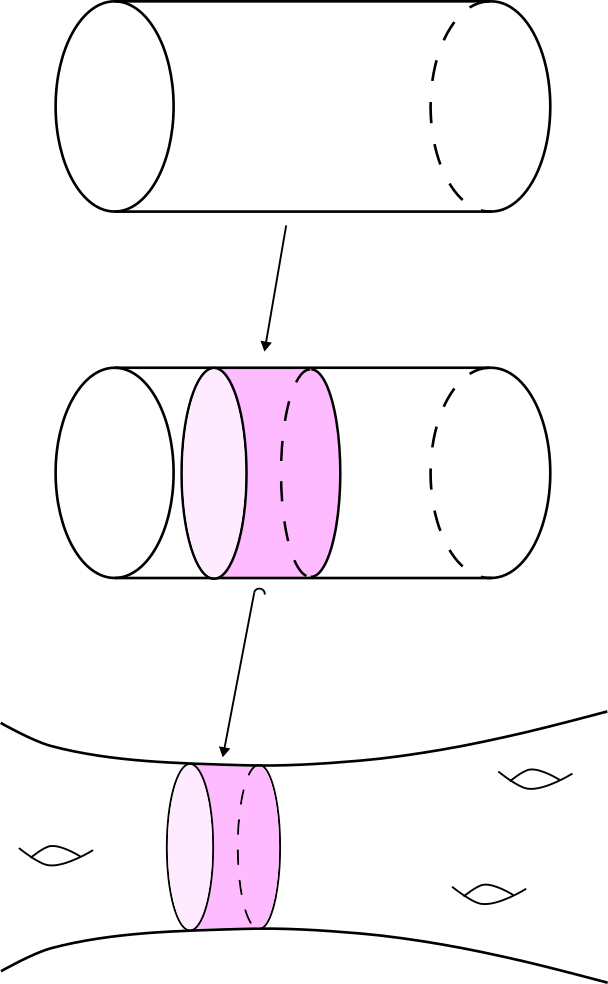}};

    \node at (6.3,9.3){$\text{Ann}$};
    \node at (3,9.3) {$\varphi \curvearrowright$};
    \node at (6.1,5.5){$A$};
    \node at (6,1.5){$\surf$};
    \node at (1,2){$\varphi_i\curvearrowright$};
    
    \node at (3.25,5.5){$A_i$};

    \node at (3.15,7.3){$f_i$};

    \node at (2.8,3.35){$\iota_i$};

    \end{tikzpicture}
\caption{A schematic of the surfaces and maps used in Section~\ref{subsec:constructinghomeo}.}\label{fig:linfinitymaps}
\end{center}
\end{figure}
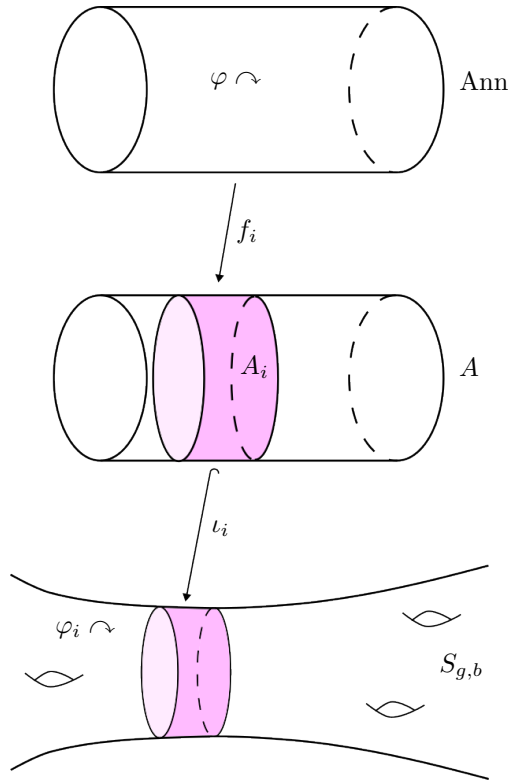

Let $\text{Ann}=S^1\times I$ be a model annulus and let $\varphi\in \Homeo_0(\text{Ann})$ be a homeomorphism isotopic to the identity fixing the boundary that acts with linear speed on $\finearc(\text{Ann})$; one is guaranteed to exist by Proposition~\ref{prop:hypactiononfinearcgraph}. 

To identify each $A_i$ with $\ann,$ for positive $i$ we define $f_i:\ann\to A_i$ by $f_i(\theta, x)=\left(\theta, \frac{x}{2^{i}}+1-
\frac{1}{2^{i-1}}\right)$ so that $f_i$ maps $[0,1]$ to $[1- \frac{1}{2^{i-1}}, 1- \frac{1}{2^{i}}]$.

Let $\iota_i:A_i\hookrightarrow \surf$ be the inclusion of $A_i$ into $\surf$. The main purpose of these inclusion mappings is to change a cylindrical coordinate of $A_i$ into a single point in $\surf.$ Abusing notation, we may refer to $\iota_i(A_i)$ by $A_i$.

Define homeomorphism $\varphi_i\in \Homeo_0(\surf)$ by
\[\varphi_i(x)= \begin{cases}
\iota_i\circ f_i \circ\varphi \circ f_i^{-1}\circ \iota_i^{-1}(x) & x\in \iota_i(A_i)\subset \surf\\
x & \text{otherwise.}
\end{cases}\]

A schematic of the above subsurfaces and maps is in Figure~\ref{fig:linfinitymaps}.

\p{Building paths} Abusing notation, we can consider $\alpha\subset \ann$ by considering $f_i^{-1}\iota_i^{-1}(\alpha\cap A_i)= \coprod_J \{\theta_j\}\times I \subset \ann$ for some finite set $J.$ 

Let $\alpha'$ be a curve isotopic to and disjoint from $\alpha$; we may assume (up to reparametrization or isotopy) that $\alpha'\cap A = \coprod_J \{\sigma_j\}\times I$ for some collection of points $\{\sigma_j\}$ in $S^1$.

Using the collection of $\theta_j$s from the previous paragraph and these $\sigma_j$'s, we apply Lemma~\ref{lemma:connectivityofextragraph} to get a canonical path $P$ of length $d$ in $\G$ between $\alpha$ and $\varphi(\alpha):$ 
\[P=(\alpha = v_0-v_1-\ldots- v_d=\varphi(\alpha))\subset \G.\]
Since homeomorphisms act isometrically on $\G,$ we can apply a power of $\varphi$ to each multiarc in the path to obtain a canonical path of length $d$ from $\varphi^m(\alpha)$ to $\varphi^{m+1}(\alpha)$ for all $m:$ 
\[\varphi^m(P) = (\varphi^m(\alpha) = \varphi^m(v_0)- \varphi^m(v_1)-\ldots- \varphi^m(v_d)=\varphi^{m+1}(\alpha))\subset \G.\]
Each such path alternates between $\theta$-multiarcs and $\sigma$-multiarcs. Moreover, by applying $\iota_i\circ f_i$ to the paths, we obtain (not properly embedded) multiarcs in $\surf.$

Given $\vec n=(n_1,n_2,\ldots,n_k, 0,0,\ldots)$ with each $n_i\in \mathbb{Z},$ we define \[\Phi=\Phi_{\vec n}=\Pi_{i=1}^k \varphi_i^{n_i}.\] Recall $|\vec n| = \max_i \{|n_i|\}$. We will now build a path from $\alpha$ to $\Phi(\alpha)$ of length $|\vec n|\cdot d.$ We define vertex (curve) $v_\ell$ along the path by defining it piecewise on the various parts of $\surf$. 

\begin{minipage}{0.8\textwidth}
\p{$\boldsymbol{A_i\subset \supp \Phi}$} Let $q^\ell = \bigcup_{i=1}^k \left(\iota_if_i \varphi^{n_i-\sign(n_i)}\right),$ where $\sign(x)=\frac{x}{|x|}.$

\p{\boldsymbol{$A\setminus\supp(\Phi)$}} If $\ell$ is odd, let $r^\ell$ be $\coprod_J \{\theta_j\}\times I$. If $\ell$ is even, let $r^\ell$ be $\coprod_J\{\sigma_j\}\times I.$

\p{\boldsymbol{$\surf\setminus A$}} If $\ell$ is odd, let $s^\ell$ be $\alpha\setminus A.$ If $\ell$ is even, let $s^\ell$ be $\alpha'\setminus A.$
\end{minipage}
\bigskip

Define a vertex $v_\ell$ along the path to be the curve $q^\ell \cup r^\ell \cup s^\ell.$ As a result, we have a path of length $d$ between $\Phi(\alpha)$ and $\Phi'(\alpha)$ where $\Phi'$ is a product of the $\varphi_i$ with the absolute values of all nonzero degrees of the $\varphi_i$ reduced by 1.

We repeat this recursively, reducing the (nonzero) absolute values of the degrees one at a time, to obtain a path of length $|\vec n|\cdot d$ between $\Phi(\alpha)$ and $\alpha.$

\subsection{Proof of Theorem~\ref{thm:linfty}}\label{subsec:prooftheoremlinfty}

Before we prove Theorem~\ref{thm:linfty}, we note that $\coo$ is quasi-isometric to the sequences of integers that eventually stabilize to 0. This is so because the inclusion of integer sequences into real sequences is quasi-surjective (each real sequence is distance less than 1 from an integer sequence) and distances in the space of integer sequences are inherited from distances in $\coo$ under the $\ell^\infty$ metric. 

\begin{proof}[Proof of Theorem~\ref{thm:linfty}]
    Let $\mathcal{A}$ be as in the statement of the theorem and fix $\alpha\in[\alpha]\in\mathcal{A}$.  For each $i\in \mathbb{N}$, let $\varphi_i\in\Homeo_0(\surf)$, and let the annuli $A$ and $A_i$ be as in Sections~\ref{sec:two} and~\ref{subsec:constructinghomeo}. Let $K'$ and $L$ be the linear action constants from Proposition~\ref{prop:hypactiononfinearcgraph}, let $d$ be the length of the path guaranteed by Lemma~\ref{lemma:connectivityofextragraph} and the work in the previous section, and let $K=\max\{K',d\}.$
    
    Let $\vec x = (x_1,x_2,\ldots) \in \coo$ and let $\vec n = (n_1,n_2,\ldots)=(\floor*{x_1},\floor{x_2},\ldots).$ For the purposes of quasi-isometry, it is enough to consider $\vec n$ in lieu of $\vec x.$ Let $N$ be an index at which the absolute values of the entries of $\vec n$ achieve their maximum, so $|n_N|=\max_i\{|n_i|\}.$ Define the following map:
    \begin{align*}
        f:\coo&\to \ourfine(\surf)\\
        \vec n &\mapsto \Pi_i \varphi_i^{n_i}(\alpha) = \Phi(\alpha).
    \end{align*}

    We claim that $f$ is a quasi-isometry. Since homeomorphisms act isometrically on fine curve graphs and fine arc graphs, to check that $f$ is a quasi-isometric embedding, it is enough to consider the distance between $\alpha$ and $\Phi(\alpha)$. Let $a$ be a single connected component of $\alpha\cap A$. Then we have the following sequence of inequalities.

    \begin{align}
        \frac{1}{K}|\vec n|-L &\leq \frac{1}{K'}|\vec n|-L \label{ineq12}\\
        &= \frac{1}{K'}|n_N|-L\label{ineq13}\\
        &\leq d_{\finearc (A_N)}(\varphi_N^{n_N}(a),a)\label{ineq14}\\
        &\leq d_{\ourfine(\surf)}(\Phi(\alpha),\alpha)\label{ineq15}\\
        &\leq d|n_N|\label{ineq16}\\
        &\leq K|n_N|+L\label{ineq17}\\
        &=K|\vec n|+L\label{ineq18}
    \end{align}

    \ref{ineq12} is true since $K'\leq K.$ 
    
    \ref{ineq13} holds by the definition of $n_N$ and the $\ell^\infty$ norm. 
    
    \ref{ineq14} holds by the definition of $\varphi$ in Section~\ref{subsec:constructinghomeo} and the properties in Proposition~\ref{prop:hypactiononfinearcgraph}.
    
    \ref{ineq15} holds since the vertices of any path $\Phi(\alpha)$ to $\alpha$  in $\ourfine(\surf)$ project to a path in $\finearc(A_N)$ and 
    
    $\Phi(\alpha)$ projects to the same thing as $\varphi_N^{n_N}(\alpha)$. 
    
    \ref{ineq16} holds by the construction of a path in Section~\ref{subsec:constructinghomeo}. \ref{ineq17} is true since $d\leq K$ and $L\geq 0.$
    
    \ref{ineq18} holds by the definition of the $\ell^\infty$ norm. 
\end{proof}

\section{Crossing Numbers}\label{sec:four}

The goal of this and the following section is to relate the distance in the single-isotopy-class fine curve graph to crossing numbers in the cyclic cover. We follow suggestions from Remark 4.6 of Bowden--Hensel--Mann--Militon--Webb \cite{BHMMW}. Our definition of the crossing number is modified  to account for more complex behavior of lifts in covers of hyperbolic surfaces (as opposed to tori). 

\begin{remark} \label{rmk:torus}
The relevant work of Bowden--Hensel--Mann--Militon--Webb concerns the fine curve graph of the torus. The fine curve graph often has a special definition in the torus case, as additional edges are added between curves that intersect exactly once. Due to these additional edges, it is more difficult to produce a lower bound for distances in this enlarged fine curve graph.

We emphasize that our definition of the fine curve graph of the torus in this and earlier sections does \emph{not} include edges corresponding to intersection number 1.
\end{remark}

For the remainder of the paper, we fix a closed hyperbolic orientable surface $S=S_{g}$, in particular, $g \geq 2$. We note our arguments apply to any compact hyperbolic surface with boundary that has connected fine curve graph. In particular, if $g=0$ then $b\geq5$ and if $g=1$ then $b \geq 2$. One modification to note is that the corresponding surface covers will also have boundary, but the same techniques work regardless.

Let $\alpha$ be an essential simple closed curve in $S$ and denote by $[\alpha]$ its isotopy class. Denote the annular cover of $S$ corresponding to $[\alpha]\in\pi_1(S)$ by $\tilde S_\alpha.$ Since $\alpha$ is a fixed representative, we will fix the precise annular cover corresponding to $\alpha.$ Let $p:\tilde S_\alpha\to S$ be the covering map.

 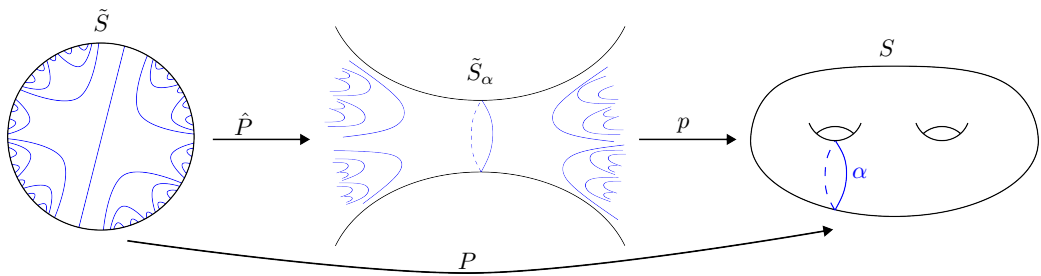
\begin{figure}[h]
\begin{center}
\begin{tikzpicture}[scale=0.9, transform shape]
    \node[anchor = south west, inner sep = 0] at (0,0) {\includegraphics[width=\textwidth]{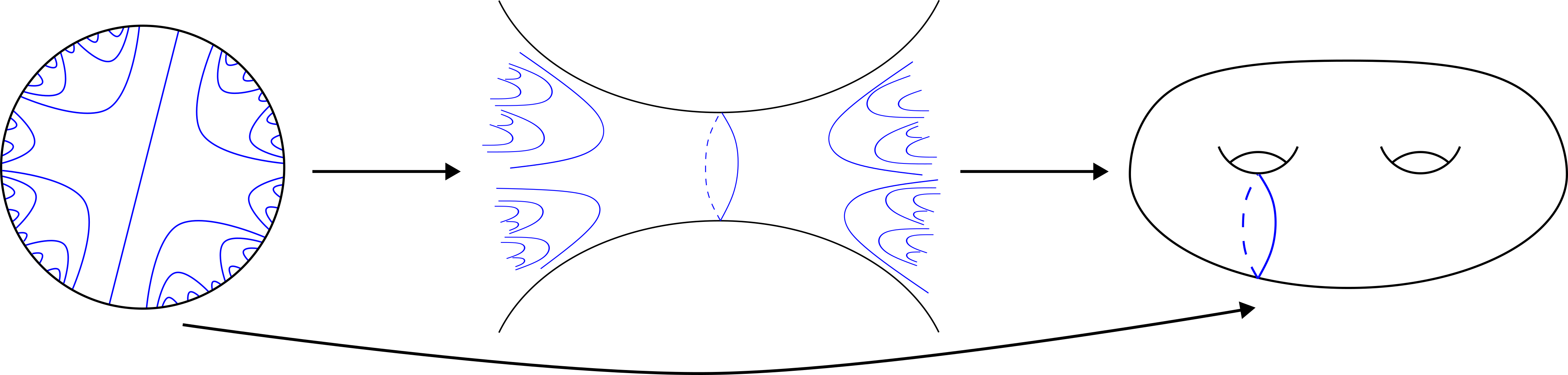}};

    \node at (12.6,1.5){\color{blue}$\alpha$};
    
    \node at (1.4,3.75){$\tilde S$};
    \node at (7,3){$\tilde S_\alpha$};
    \node at (13,3.35){$S$};

    \node at (3.5,2.23){$\hat P$};
    \node at (10,2.2){$p$};
    \node at (6.8,0.2){$P$};

    \end{tikzpicture}
\caption{The relevant covers and maps used throughout Sections~\ref{sec:four} and \ref{sec:five}.}\label{fig:covers}
\end{center}
\end{figure}

 Let $\tilde S$ be the universal cover of $S$ (the hyperbolic plane) and let $P:\tilde S \to S$ be the covering map. Let $\hat P: \tilde S \to \tilde S_\alpha$ be the intermediate covering map, so $P=p\circ \hat P$. If a statement applies to both $\tilde S$ and $\tilde S_\alpha,$ we will write $\tilde S_{(\alpha)}.$ Unless otherwise stated, we will be working in the annular cover.
 
Generally, we will have curves $\beta$ and $\gamma$ isotopic to $\alpha.$ We may refer to $\tilde S_\alpha$ by $\tilde S_\gamma$ as well. Since we will be focused on the topology rather than the geometry of the cyclic cover, this distinction does not matter. However, for visualization purposes, we note that since we restrict our work to hyperbolic surfaces, we may equip $S$ with a hyperbolic metric, making $\tilde S_\alpha$ a flared annulus.

Now we explore intersection patterns of curves in $\tilde S_\alpha$ with the goal of bounding distances in single-isotopy-class fine curve graphs. Let $\beta,\gamma\in[\alpha]$. We will bound $\dfiber(\beta,\gamma)$, the distance between $\beta$ and $\gamma$ in the subgraph of $\fine(S)$ induced by all curves homotopic to $\alpha,$ using the \emph{crossing numbers} between $\beta$ and $\gamma$.

Given a curve $\gamma \in [\alpha],$ an \emph{elevation} of $\gamma$ in $\tilde S_\alpha$ (or $\tilde S$) is a connected component of $p^{-1}(\gamma)$ (or $P^{-1}(\gamma)$). Moreover, exactly one elevation of $\gamma$ in $\tilde S_\alpha$ is compact, and it is isotopic in $\Tilde{S}_\alpha$ to the unique compact geodesic. We denote this compact elevation by $\hat{\gamma}.$

\p{Tree of elevations} We define a tree $T_\gamma$ corresponding to an oriented curve $\gamma \in [\alpha]$ by the following. Let the elevations of $\gamma$ be the vertices of $T_\gamma$. We connect two elevations with an edge exactly when these elevations are in the closure of a common connected component of $\Tilde{S}_\alpha \setminus p^{-1}(\gamma)$.%

\begin{figure}[h]
\begin{center}
\begin{tikzpicture}
    \node[anchor = south west, inner sep = 0] at (0,0) {\includegraphics[width=0.6\textwidth]{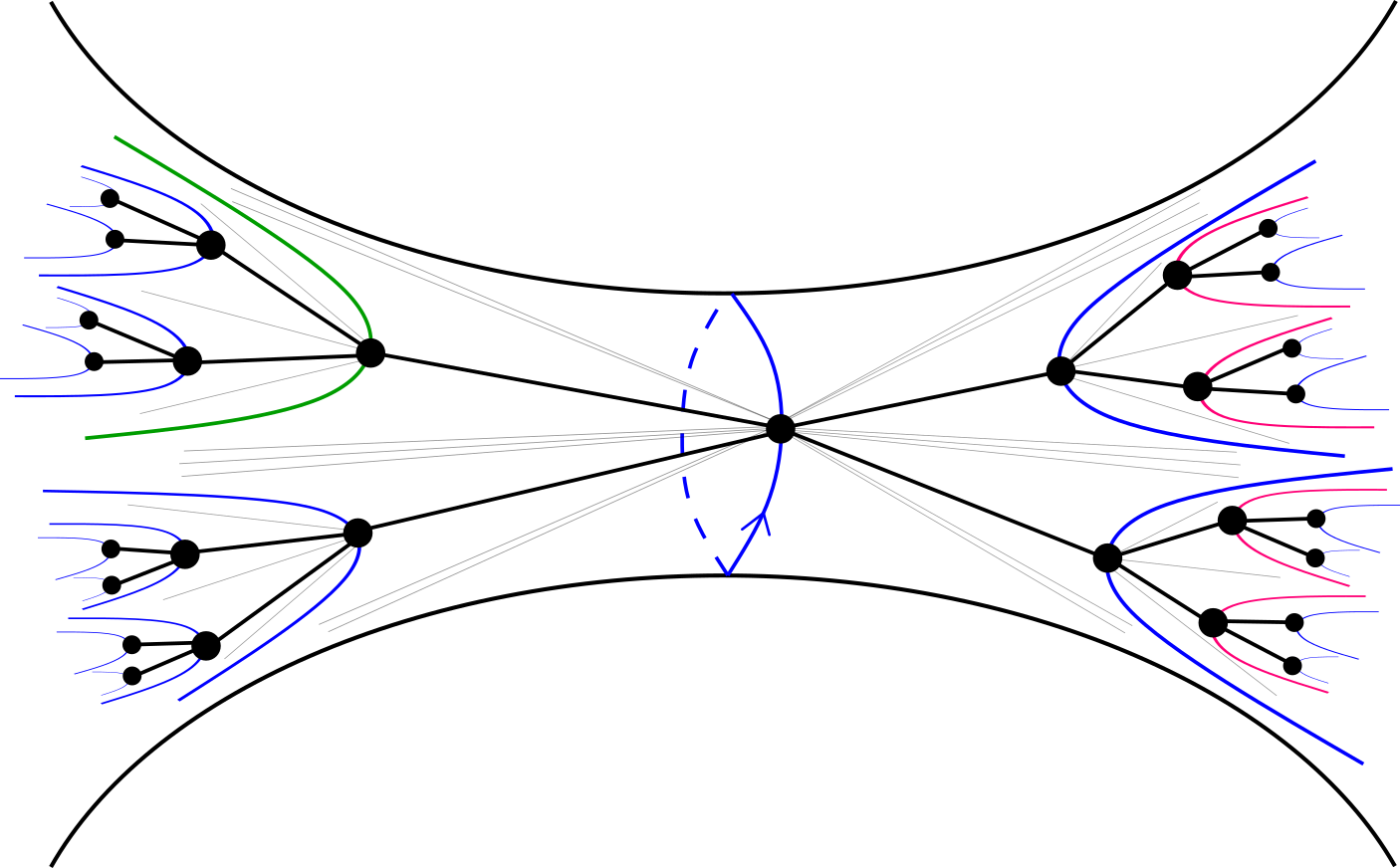}};
    
    \node at (4.7,4){0};
    \node at (4.7,1.7){$\hat{\gamma}$};
    \node at (2.3,4.4){$-1$};
    \node at (7,4.4){1};
    \node at (1.3,4.9){$-2$};
    \node at (7.7,4.7){2};

    \end{tikzpicture}
\caption{An example of how to construct a tree $T_\gamma$ of elevations of a curve $\gamma.$ All level 2 lifts of $\gamma$ are pink. One elevation of $\gamma$ at level $-1$ is green.}\label{fig:deflevelselevations}
\end{center}
\end{figure}
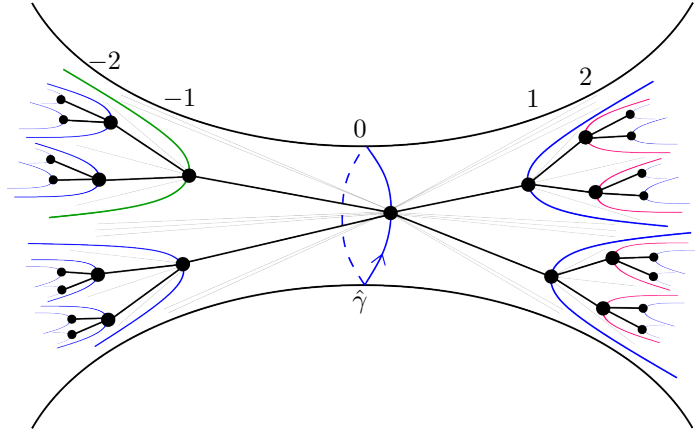

Let $\hat{\gamma}$ be the compact elevation of $\gamma.$ Fixing an orientation of $\Tilde{S}_\alpha$, we define the \emph{level} of an elevation as the distance in $T_\gamma$ of its corresponding vertex $v$ to $\hat{\gamma}$ with positive or negative value depending on whether $v$ is to the right or left of $\hat{\gamma}$. We refer to a level of $\gamma$ as the union of all elevations of the same level.

The tree of elevations can also be defined for the universal cover with the level of an elevation defined as the level of the elevation in the cyclic cover that it covers. Note the covering map $\hat P:\tilde S\to \tilde S_\alpha$ preserves levels and setwise fixes a unique elevation that covers the compact elevation.

\p{Crossing number}
    Let $\beta$ and $\gamma$ be isotopic curves. Define the \emph{$\gamma$-crossing number of $\beta$}, denoted $C_{\gamma}(\beta)$, to be the number of levels of $\gamma$ that $\hat{\beta}$ intersects. Note this is always finite since $\hat{\beta}$ is compact. When $\beta=\gamma$, we set the crossing number to zero. Since $\beta$ is isotopic to $\gamma$, we can pair each elevation of $\beta$ with an elevation of $\gamma$ such that the isotopy between the curves lifts to an isotopy between the elevations.

\begin{lemma} \label{lem:homotopiclift}
    Let $S=S_g$ with $g\geq 2$ and let $\beta$ and $\gamma$ be homotopic essential simple closed curves. Then any free homotopy from $\beta$ to $\gamma$ can be lifted to a free homotopy from any given elevation of $\beta$ to a unique elevation of $\gamma$ in $\tilde S_{(\gamma)}$. 
\end{lemma}

\begin{proof}
    We will view a curve in $S$ as the image of a map $\mathbb{R}$ to $S$ with the image of $x$ equal to the image of $x+1$ (amounting to quotienting $\mathbb{R}$ to create $S^1$). Thus, where appropriate, we may consider a curve to be a map $S^1\to S.$ 

    We will first prove that a free homotopy from $\beta$ to $\gamma$ can be lifted to a homotopy between $\hat\beta$ and $\hat\gamma$, the compact elevations of $\beta$ and $\gamma$ in $\tilde S_\gamma$. This can be done via usual homotopy lifting.

    Next we consider noncompact elevations of $\beta$ and $\gamma$ in both $\tilde S$ and $\tilde S_\gamma.$ When viewing $\beta$ and $\gamma$ as images of $\mathbb{R}$ in $S,$ we have a free homotopy between them being a map $\mathbb{R}\times I\to S.$ Since a noncompact elevation of $\beta$ (or $\gamma$) is a lift of $\beta$ (or $\gamma$)---and amounts to a map $\mathbb{R}\to \tilde S_{(\gamma)}$---we may apply the homotopy lifting property to obtain a homotopy from an elevation of $\beta$ to an elevation of $\gamma.$ Since a homotopy in $\tilde S_{(\gamma)}$ fixes endpoints in the boundary, the elevation of $\gamma$ is unique.
\end{proof}

We call two elevations paired by Lemma~\ref{lem:homotopiclift} \textit{corresponding elevations}. By orienting $\beta$ and $\gamma$ so that the isotopy preserves this orientation, we can also pair the levels of $\beta$ with levels of $\gamma$.

\p{Arcs between elevations} Define an \emph{arc} in $\tilde S_\alpha$ to be the image of an embedding $[0,1]\hookrightarrow \tilde S_\alpha.$ We will mostly consider arcs with endpoints in $p^{-1}(\gamma)$ with $\gamma\in [\alpha].$ We say an intersection between two curves or arcs is \emph{transverse} if they cross at the point of intersection. Two curves or arcs are transverse if they are transverse at each point of intersection. 

We call an arc in $\tilde S_\gamma$ \emph{degenerate} if it is homotopic (relative its endpoints) to a subset of $p^{-1}(\gamma).$ Otherwise, an arc is \emph{nondegenerate}. A nondegenerate arc $\zeta$ in $\tilde S_{(\alpha)}$ is in \emph{minimal position} with $\gamma$ if $\zeta$ is transverse to $p^{-1}(\gamma)$ (or $P^{-1}(\gamma)$) and does not form any bigons with $p^{-1}(\gamma)$ (or $P^{-1}(\gamma)$). A degenerate arc $\zeta$ is in minimal position with $\gamma$ if $|\zeta\cap p^{-1}(\gamma)|=1$ or $2$ depending on whether the endpoints agree or not (similarly for $|\zeta\cap P^{-1}(\gamma)|$).

\begin{lemma} \label{lemma: essential1}
    Let $\zeta$ be a nondegenerate arc in $\tilde S$ in minimal position with $\beta.$ Then $\zeta$ intersects each elevation of $\beta$ at most once.
\end{lemma}

\begin{proof}
   
   This follows from a standard innermost bigon argument and our definition of minimal position (see the proof of \cite[Lemma 1.8]{primer}, for example). \end{proof}

\begin{figure}[h] 
\begin{center}
\includegraphics[width=0.5\textwidth]{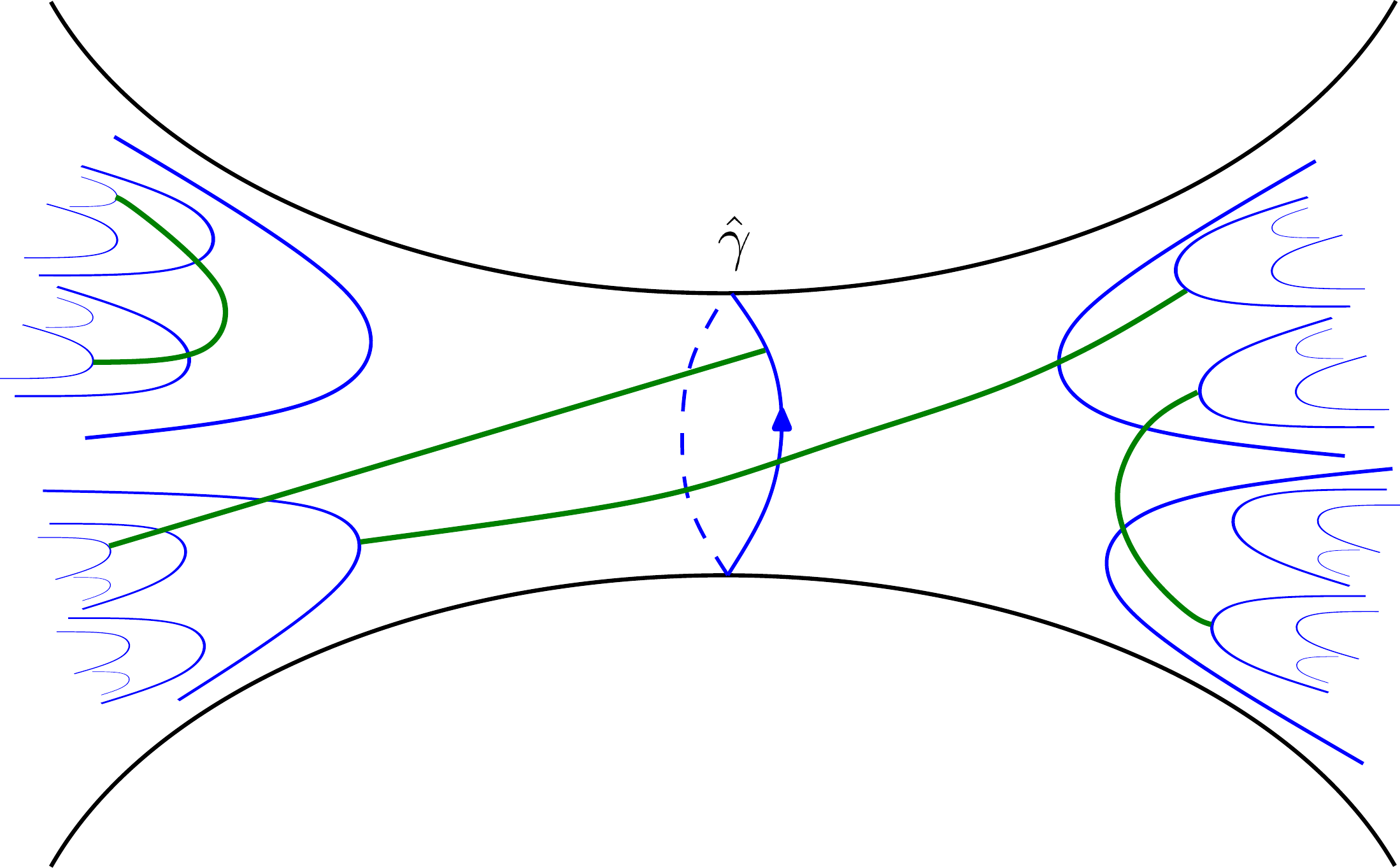}
\caption{The green arcs are in minimal position arcs and project to the same arc in the surface.} \label{fig:arcslevels}
\end{center}
\end{figure}

We note that Lemma~\ref{lemma: essential1} is \emph{not} true in $\tilde{S}_\gamma$; see Figure~\ref{fig:arcsameelevation} for an example. The simpler configurations in the universal cover are why we often factor through $\tilde S$ in our proofs. We further record the following lemma stating that every arc is homotopic (rel endpoints) to a minimal position arc.

\begin{lemma}\label{lemma:archomminpos}
    Let $S=S_g$ be a surface with $g\geq 2$ and let $\gamma$ be a curve in $S$ and $\zeta$ an arc in $\tilde S_{(\gamma)}$ with endpoints in lifts of $\gamma$. Then $\zeta$ is homotopic relative endpoints to an arc $\zeta'$ in minimal position with $\gamma.$
\end{lemma}

\begin{proof}
    This follows from well-known surface topology. See for example the ideas in Shapiro \cite{shapiro} and the bigon criterion in \emph{A Primer on Mapping Class Groups} \cite{primer}.
\end{proof}

 \pit{Intersection list of an arc} We can orient $\zeta$ and record elevations of $\gamma$ that $\zeta$ intersects by the ordered list $L=(l_1,l_2,\ldots,l_n).$ For example, if $\hat{\gamma}$ is oriented up at the front in Figure~\ref{fig:arcslevels}, the arcs in the figure have sample lists $(-3,-2,-2,-3),$ $(-3,-2,-1,0),$ $(-1,0,1,2),$ and $(2,1,1,2).$ The following lemma restricts the intersection lists an arc in minimal position can have.
 
\begin{lemma}\label{lemma:hitscompact}
    Let $S=S_g$ be a surface with $g\geq 2$ and let $\gamma$ be a curve in $S$ and $\zeta$ an arc in $\tilde S_{(\gamma)}$ in minimal position with $\gamma$ with endpoints in levels $l_1$ and $l_n$. Let $L = (l_1, l_2, \ldots, l_n)$ be an ordered list of levels of the elevations that $\zeta$ intersects. Then there are two options for $L$:
    \begin{enumerate}
        \item $L$ is strictly monotonic and $|l_j-l_{j+1}|=1$ for all $1\leq j<n,$ or
        \item there exists $k\in\{1,\ldots,n-1\}$ such that:
        \begin{enumerate}
            \item $l_k=l_{k+1},$
            \item $(|l_1|,\ldots,|l_k|)$ is strictly decreasing while $(|l_{k+1}|,\ldots,|l_n|)$ is strictly increasing, and
            \item $|l_j-l_{j+1}|=1$ for $j\neq k.$
        \end{enumerate}
        It follows that $(l_1,\ldots,l_k)$ and $(l_{k+1},\ldots, l_n)$ are both strictly monotonic.
    \end{enumerate}
    Moreover, if $\zeta$ is nondegenerate and $l_j=0$ for some $j\in\{1,\ldots,n\},$ then option (1) holds.
\end{lemma}

\begin{proof} 
It suffices show the result for the universal cover since the quotient map $\hat P:\tilde S\to \tilde S_\gamma$ preserves levels (since $\hat P$ is defined by translation along the level 0 elevation of $\gamma$) and an embedded arc is in minimal position with $\gamma$ in $\tilde S_\gamma$ if and only if its lifts are in minimal position with $\gamma$ in $\tilde S$ (see the proof of the bigon criterion in \emph{A Primer on Mapping Class Groups} \cite{primer}). 

If $\zeta$ is a degenerate arc, then its list of levels in minimal position is $(l,l),$ which satisfies the conditions of the lemma. Otherwise, the endpoints of $\zeta$ are on different elevations of $\gamma,$ so each elevation of $\gamma$ is intersected at most once by Lemma \ref{lemma: essential1}. 

Let $D$ be a connected component of $\tilde S\setminus P^{-1}(\gamma)$. Note $D\cap \zeta$ must be either empty, an arc between elevations differing by one level, or an arc connecting two different elevations of $\gamma$ at the same level. In the case of repeated levels, the repeated level cannot be 0 since there is only one elevation at level 0. In the other case, the absolute values of the levels of the endpoints for $D\cap \zeta$ differ by 1.

If some $D\cap \zeta$ is an arc connecting two different elevations at the same level, then continuing along $\zeta$ in either direction to an adjacent component, $\zeta$ will next hit levels with higher absolute value. This gives the first part of the lemma.

If $\zeta$ is nondegenerate and intersects level 0 of $\gamma$, applying this same argument shows that $\zeta$ is strictly monotonic, so we are now done. \end{proof}

\begin{remark}
    It follows from Lemma~\ref{lemma:hitscompact} that any arc with an endpoint on level 0 is homotopic to a minimal position arc in the cyclic cover $\tilde S_\gamma$ that hits each level (and therefore elevation) at most once. Otherwise, the minimal position arc may repeat a level and may even intersect the same elevation twice, as in Figure ~\ref{fig:arcsameelevation}.
\end{remark}

\begin{figure}[h] 
\begin{center}
\includegraphics[width=0.5\textwidth]{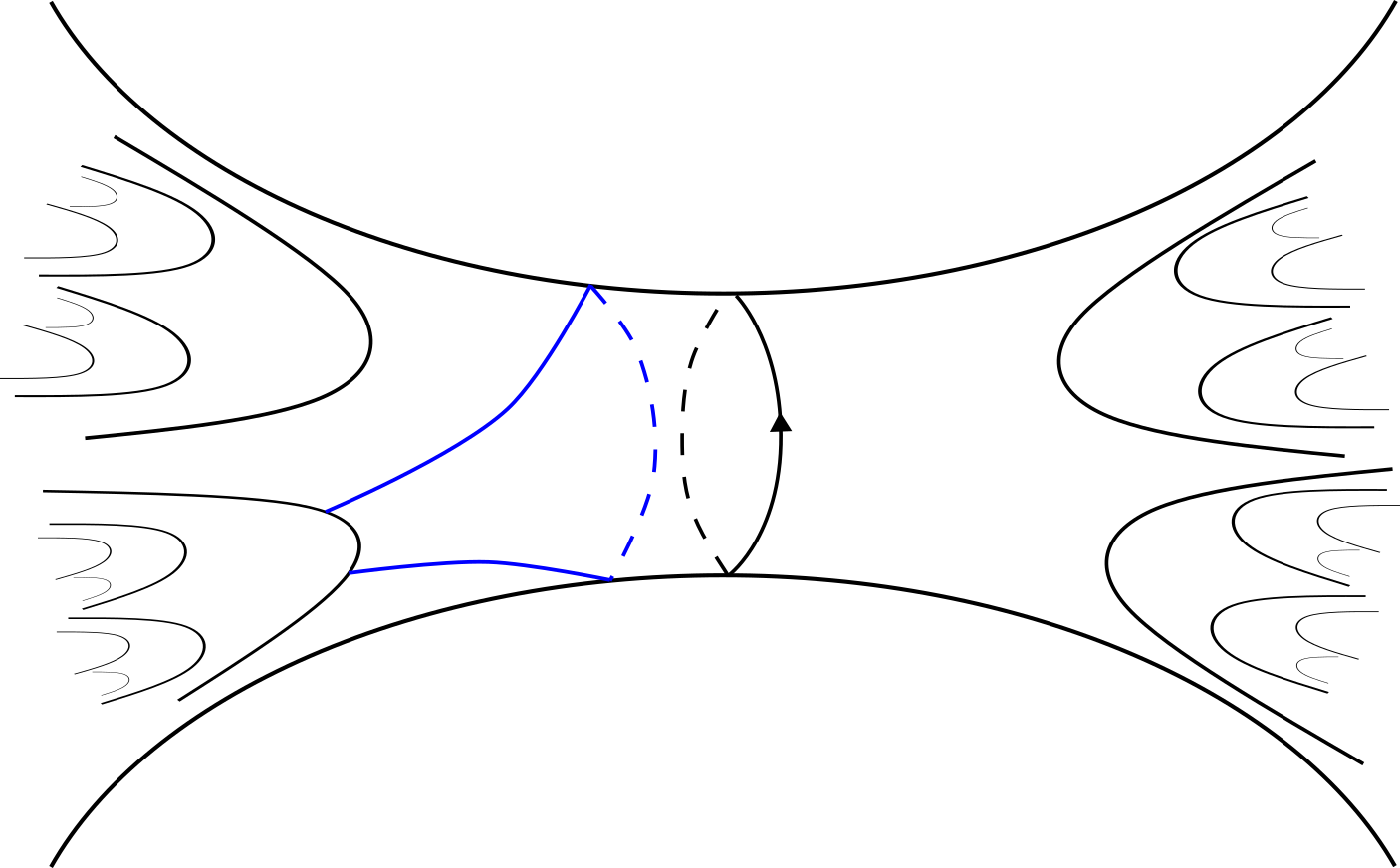}
\caption{An arc in minimal position with $\gamma$ that intersects the same elevation of $\tilde S_\gamma$ twice. The list of levels for the arc is $(-1,-1).$} \label{fig:arcsameelevation}
\end{center}
\end{figure}

 We say an arc $\zeta$ in $\tilde S_{(\gamma)}$ between elevations of $\gamma$ \textit{essentially intersects} $n$ elevations if it is homotopic relative endpoints to an arc in minimal position that intersects $n$ elevations, including at the endpoints.

\begin{lemma}  \label{lemma: essential2}
    Let $\zeta$ be an arc in the cyclic cover with an endpoint on a level $\pm 1$ elevation of $\beta$. Then $\zeta$ essentially intersects $n$ elevations of $\beta$ if and only if at least one of the following holds:

    \begin{enumerate}
        \item $\zeta$ has its other endpoint on a level $\pm n$ elevation, 
        \item $\zeta$ has its other endpoint on a level $\mp(n-2)$ elevation, or
        \item $\zeta$ has its other endpoint on a level $\pm (n-1)$ elevation, and it essentially intersects one other level $\pm 1$ elevation.
    \end{enumerate}
\end{lemma}

\begin{proof}
    Without loss of generality, suppose $\zeta$ is in minimal position with $\beta$ (Lemma~\ref{lemma:archomminpos}) and has an endpoint at level 1. By Lemma~\ref{lemma:hitscompact}, either 1) $\zeta$ increases monotonically upwards in levels, 2) $\zeta$ increases monotonically downwards in levels, 3) $\zeta$ repeats level $1$ and then increases monotonically. Each case corresponds to the similarly numbered case in the lemma statement.
\end{proof}

We define the \textit{translates} of a connected subset $\zeta$ in the cyclic cover or universal cover as any connected subset such that projection to the surface is injective and agrees with the projection of $\zeta$. We use the following observations.

\begin{lemma}\label{lemma: arcs}
    Let $\beta$ and $\gamma$ be homotopic curves.
    \begin{enumerate}
        \item[(1)] If $\zeta$ is an arc between corresponding elevations of $\beta$ and $\gamma$ in $S_{(\gamma)},$ then every translate of $\zeta$ is an arc between corresponding elevations.
        \item[(2)] If $\zeta$ is an arc essentially intersecting $n$ elevations of $\beta$ in $S_{(\gamma)}$, then any translate of $\zeta$ essentially intersects $n$ elevations of $\beta$. 
    \end{enumerate}
    In particular, (1) holds if $\zeta$ is a single point.
\end{lemma}

\begin{proof}
    We begin with several observations. First, the projection map $\hat P:\tilde S\to \tilde S_\gamma$ takes pairs of corresponding elevations in the universal cover to pairs of corresponding elevations. 

    Second, essential intersection number is preserved by $\hat P$ (this is the bigon criterion). In combination with the first observation, this implies that we can prove (1) and (2) for annular covers by lifting, performing the proof in the universal cover, and then projecting back down. Thus we will prove (1) and (2) only in the universal cover.

    (1) follows from the facts that: (a) every deck transformation in $\tilde S$ takes corresponding elevations of $\beta$ and $\gamma$ to corresponding elevations since it acts by homeomorphisms on $\tilde S$ and homeomorphisms respect boundary points, and (b) deck transformations in $\tilde S$ act transitively on lifts of $\gamma$.

    (2) follows from the observation that since deck transformations are homeomorphisms, they preserve essential intersection number of an arc with $\beta$ or $\gamma$.
\end{proof}

\section{Symmetry of crossing numbers}\label{sec:five} 

A natural question is to ask whether crossing number is symmetric. Note the crossing number could be defined using the number of \textit{elevations} intersected (with multiplicity) rather than the number of levels, and this gives a symmetric quantity. However, we chose to define crossing number using levels to give a more robust distance bound in later propositions. We will show that the crossing number is symmetric if $\alpha$ is separating, but is not necessarily symmetric if $\alpha$ is nonseparating. 

    \p{Crossing number bounds for nonseparating curves} We begin with an explicit example demonstrating that crossing number is not symmetric. 
    
\begin{example}
    In Figure \ref{fig:cross3}, we present nonseparating curves $\beta$ and $\gamma$ such that $C_{\gamma}(\beta) = 3$ but $C_{\beta}(\gamma) = 2$. In both cases, a compact elevation for a chosen curve hits exactly three elevations of the other curve, but the compact elevation $\hat{\gamma}$ meets two elevations at level 1.
    
    The curve $\beta$ in this example was constructed by applying two point push maps on a curve parallel to $\gamma.$ The point push maps were chosen so that one begins by pushing $\beta$ to the left and the other to the right until it hits $\gamma$.  This ensures the final result $\beta$ must lift to a compact elevation hitting levels $-1$ through 1. However, the point push maps were chosen to push $\beta$ in the same direction through $\gamma$, so any elevations of $\beta$ that intersect $\hat{\gamma}$ have to come from the same side of $\hat{\gamma}$. A more detailed explanation of how these curves are constructed can be seen in Figure~\ref{fig:crossn}.
\end{example}

    Expanding on this example, we have the following bounds on crossing number.    

    \begin{lemma} \label{prop:nonsepcurve} Let $S=S_g$ with $g\geq 2$. Two isotopic nonseparating curves $\beta$ and $\gamma$ are either disjoint or  
        \[\lfloor \frac{\max\{C_{\gamma}(\beta), C_{\beta}(\gamma)\}}{2}\rfloor +1   \leq \min\{C_{\gamma}(\beta), C_{\beta}(\gamma)\}.\]
    \end{lemma} 

           \begin{figure}[h]
\begin{center}
\includegraphics[width=0.6\textwidth]{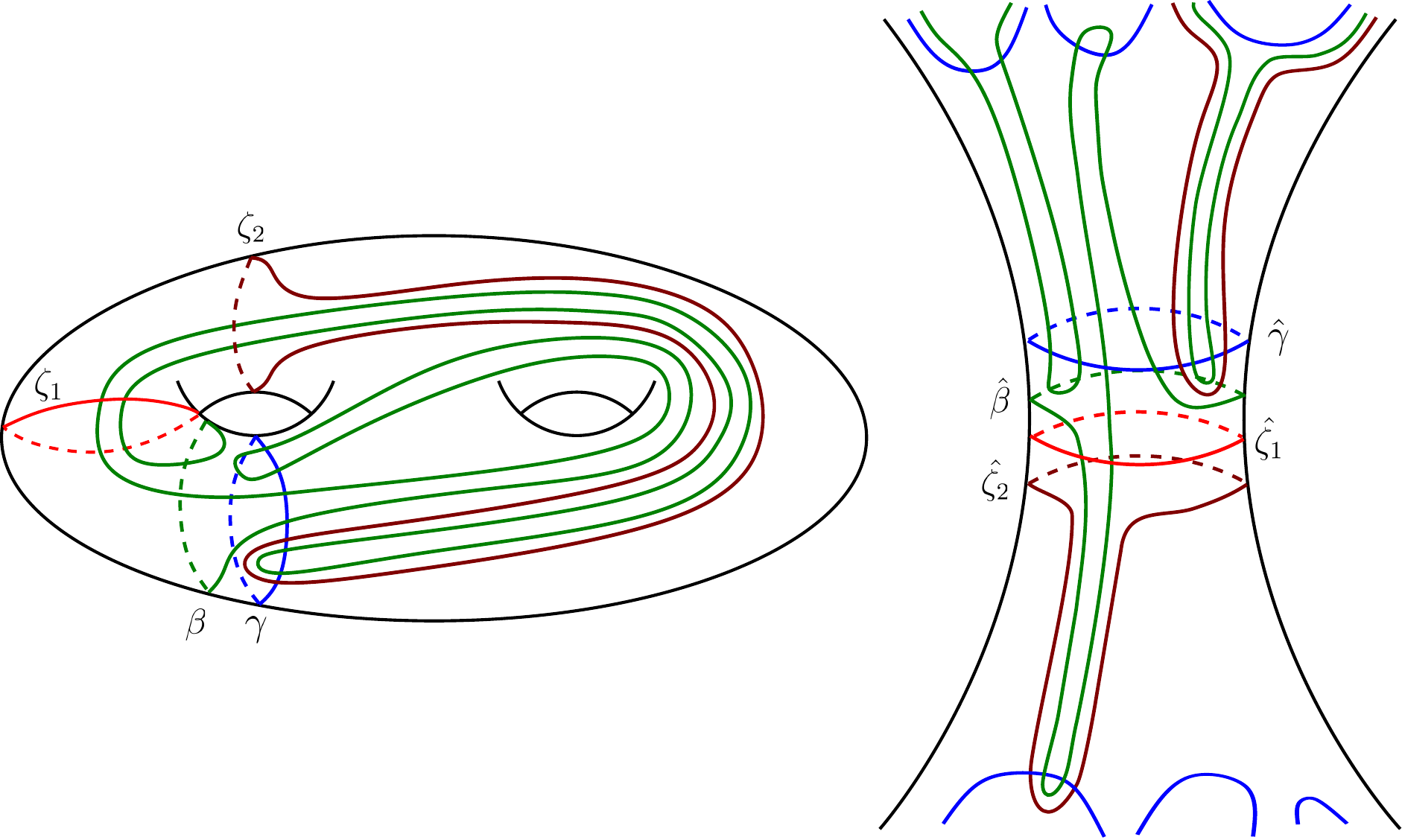}
\caption{Left: $\beta$ and $\gamma$ are distance 3 curves, and a path between them is $\gamma-\zeta_1-\zeta_2-\beta.$ Right: we see the lift of the curves on the left to the cyclic cover. Notice that $C_\gamma(\beta)=3$ while $C_\beta(\gamma)=2.$}
\label{fig:cross3}
\end{center}
\end{figure}

    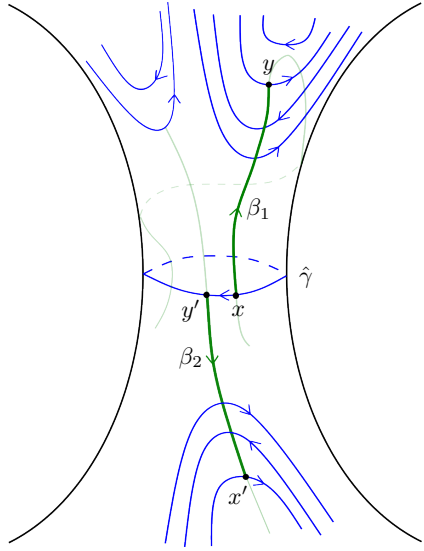
\begin{figure}[h]
\begin{center}
\begin{tikzpicture} [scale=0.9, transform shape]
    \node[anchor = south west, inner sep = 0] at (0,0) {\includegraphics[width=0.4\textwidth]{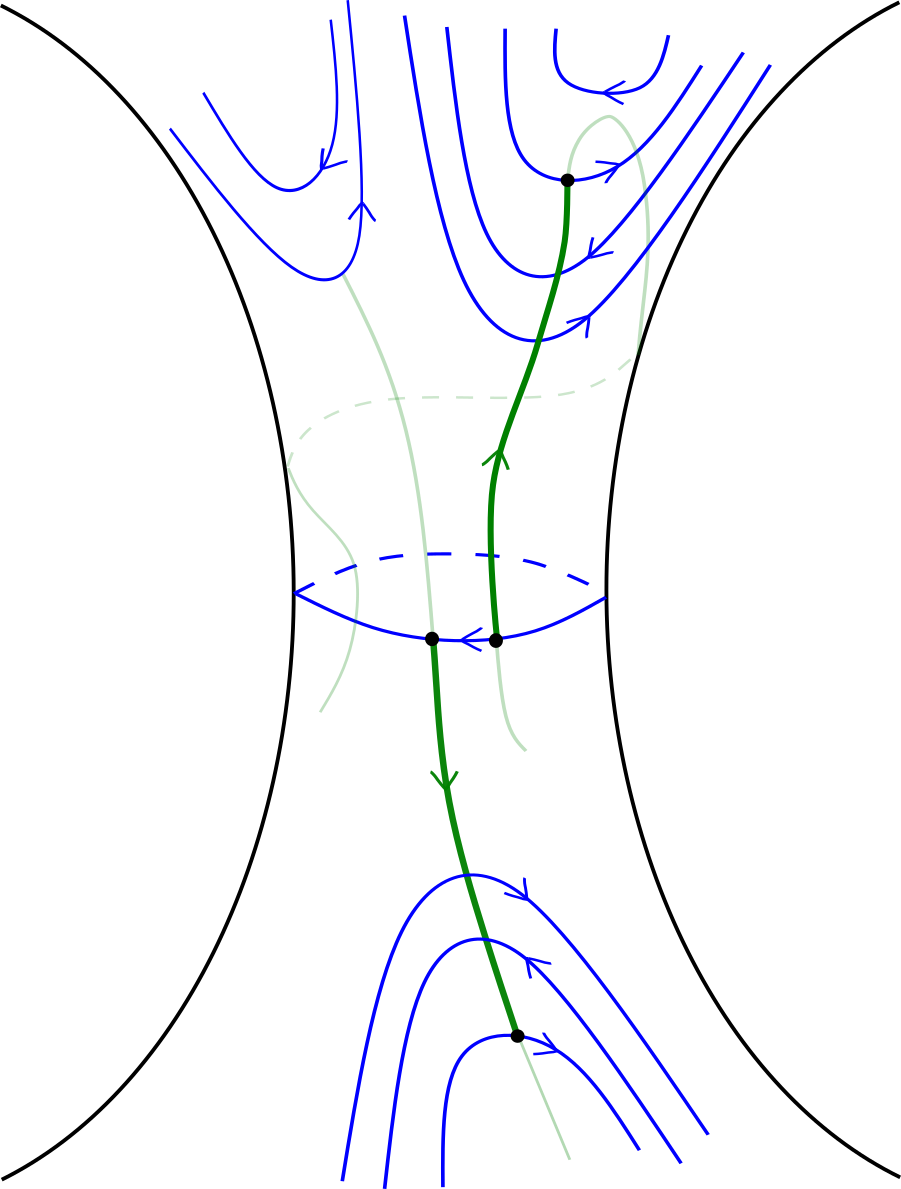}};
    
    \node at (4.4,4){$\hat\gamma$};
    \node at (3.4,3.5){$x$};
    \node at (3.7,5){$\beta_1$};
    \node at (3.85,7.1){$y$};
    \node at (2.7,3.5){$y'$};
    \node at (3.4,0.77){$x'$};
    \node at (2.7,2.8){$\beta_2$};

    \end{tikzpicture}
\caption{A schematic of all arcs and points used in the proof of Case 1 of Lemma~\ref{prop:nonsepcurve}}\label{fig:4.7}
\end{center}
\end{figure}

\begin{SCfigure}[0.7][htp] 
\begin{tikzpicture}
    \node[anchor = south west, inner sep = 0] at (0,0) {\includegraphics[width=0.5\textwidth]{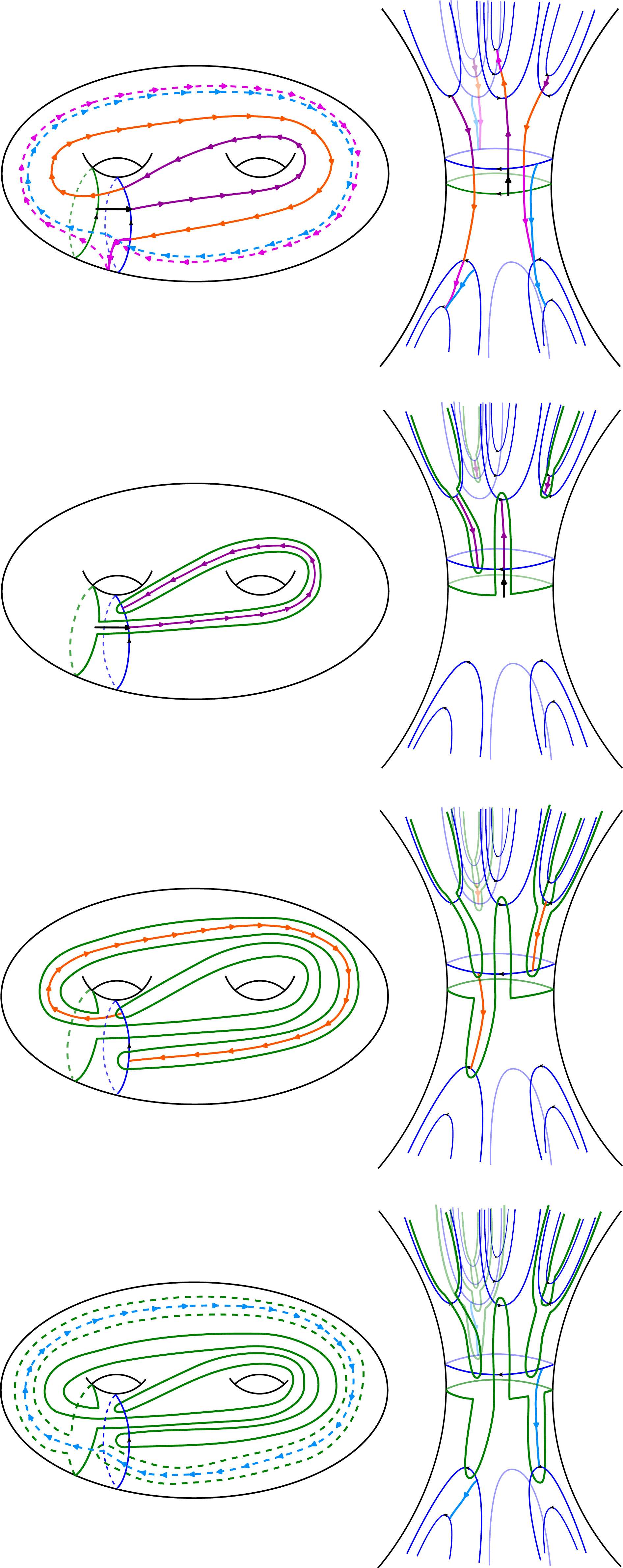}};
    
    \node at (7,16.9){$\hat\beta$};
    \node at (7,17.45){$\hat\gamma$};
    \node at (7,12){$\hat\beta$};
    \node at (7,7.05) {$\hat\beta$};
    \node at (7,2.15){$\hat\beta$};

    \end{tikzpicture}

\caption[]{
This is the beginning (Steps 0) of a sequence of point pushes that allows us to achieve the lower bound in Lemma~\ref{prop:nonsepcurve}. More descriptions can be found in Table~\ref{table:table}. \medskip
\par In the top picture, we show the arcs that will be pushed along throughout. Since $\beta$ and $\gamma$ are disjoint, $(C_\gamma(\beta),C_\beta(\gamma))=(0,0).$
\vspace{1.1cm}
\par  In the second picture, we show the creation of the first tail by pushing along the black and purple arcs. In the right hand side, we see that there is a level 1 elevation of $\beta$ coming down to intersect $\hat\gamma$, so $(C_\gamma(\beta),C_\beta(\gamma))=(2,2).$
\vspace{1.2cm}
\par In the third picture, we show the creation of the second tail by pushing along the orange arc. A tail of $\hat\beta$ is reaching down to level $-1$ of $\gamma$ for the first time, thus increasing $C_\gamma(\beta)$ by 1. A new level 1 elevation of $\beta$ is coming down to intersect $\hat\gamma$, but that does not influence $C_\beta(\gamma)$ since it is not a new level (just a new elevation). As a result, $(C_\gamma(\beta),C_\beta(\gamma))=(3,2).$ 
\vspace{1cm}
\par
 In the fourth picture, we show the creation of the third tail by pushing along the blue arc. Since no new levels of $\beta$ nor $\gamma$ are intersected, $(C_\gamma(\beta),C_\beta(\gamma))=(3,2)$ once more. \vspace{1.7cm}
}
\label{fig:crossn}
\end{SCfigure}

    \begin{proof} Let $n$ be the maximum of the two crossing numbers, and without loss of generality say $C_{\gamma}(\beta) = n$. That is, the compact elevation of $\beta$, denoted $\hat{\beta}$, intersects $n$ levels of $\gamma$. Likewise, let $\hat{\gamma}$ be the compact elevation of $\gamma$. We are done if we show $\hat{\gamma}$ intersects at least $\lfloor \frac{n}{2}\rfloor +1$ levels of $\beta$.
    
    \pit{Case 1: compact elevations intersect} Suppose $\hat{\beta}$ and $\hat{\gamma}$ intersect. Since $\hat{\beta}$ intersects $n$ levels of $\gamma$ including the compact elevation at level 0, there are $n-1$ intersected non-0 levels of $\gamma$. Some of these are positive and others are negative. By the pigeon hole principle, $\hat{\beta}$ intersects level $\pm \lceil\frac{ n-1 }{2}\rceil=\pm \lfloor \frac{n}{2}\rfloor$ of $\gamma$. Without loss of generality, it intersects level $m = \lfloor\frac{n}{2}\rfloor$.

   We will produce a level $\pm m$ lift of $\beta$ that intersects $\hat \gamma.$ Let $\beta_1\subset \hat\beta$ be an arc with endpoints on levels 0 and (without loss of generality) $m$ of $\gamma.$ Call these endpoints $x$ and $y,$ respectively.
    
    Let $y'\in \hat\gamma \cap p^{-1}(p(y)).$ There is a unique lift of $p(\beta_1)$ with an endpoint at $y'$; call this lift $\beta_2.$ Let $x'\in p^{-1}(p(x))$ be the other endpoint of $\beta_2.$  

    Since $\beta_1$ is an arc from level 0 to level $m$ of $\gamma$, it essentially intersects $m+1$ elevations of $\gamma$ by Lemma~\ref{lemma:hitscompact}. Since $\beta_2$ is another lift of the same arc as $\beta_1,$ it must also essentially intersect $m+1$ elevations of $\gamma$ by Lemma~\ref{lemma: arcs}. 

    Since $x$ is an intersection of compact lifts of $\beta$ and $\gamma,$ by Lemma~\ref{lemma: arcs}, $x'$ must be the intersection of two lifts of $\beta$ and $\gamma$ at the same level: both $m$ or $-m.$ 
    We conclude that $\beta_2$ must be a subarc of a lift of $\beta$ at level $\pm m$, so $\hat\gamma$ intersects level $\pm m$ of $\beta.$

    Since $\hat\gamma$ intersects levels 0 through $\pm m$ of $\beta,$ it intersects at least $m+1= \lfloor \frac{n}{2}\rfloor +1$ levels of $\beta,$ as desired.

\pit{Case 2: compact elevations are disjoint} When the compact elevations do not intersect, we can modify the previous proof by attaching $\hat{\beta}$ and $\hat{\gamma}$ with an arc $\zeta$. By concatenating $\zeta$ with an arc in $\hat{\beta}$, we get an arc $\beta_1$ from level 0 to level $\pm n$. Note that any other lift of $p(\zeta)$ to the cyclic cover is an arc between corresponding elevations by Lemma~\ref{lemma: arcs}. A similar argument to Case 1 shows we have a lift of $p(\beta_1)$ from level 0 to level $\pm n$, so $C_\gamma(\beta)=n$. 
\end{proof}

\begin{lemma}\label{lemma:achievedminimum}
    The bound in Lemma \ref{prop:nonsepcurve} is optimal.
\end{lemma}

\begin{proof}
We give an explicit construction in genus two that extends to higher genus surfaces by attaching handles. We are inspired by the curves in Figure~\ref{fig:cross3}. In particular, we produce pairs of curves $\beta$ and $\gamma$ with crossing numbers $(C_\gamma(\beta),C_\beta(\gamma))=(2n-2,n)$ and $(C_\gamma(\beta),C_\beta(\gamma))=(2n-1,n)$ for all $n\geq 2.$ 

The above is accomplished by applying a sequence of point pushes to the curves $\beta$ and $\gamma$ pictured in the top left image of Figure~\ref{fig:crossn}. (The arcs are shown in isolation in the remainder of Figure~\ref{fig:crossn}.) These point pushes occur in small neighborhoods of the pictured arcs (black, purple, orange, pink, and blue), as in the top left of Figure~\ref{fig:crossn}. The main ideas are that: (1) following every point push, the changed parts of $\beta$ are in a small neighborhood of the arcs pushed along both in $S$ and in $\tilde S_\gamma$, (2) each push is along a simple arc whose lifts are in minimal position with $\gamma$ in $\tilde S_\gamma,$ and (3) the pushes are chosen in a way such that the resulting curves are always simple.

\pit{Describing the point push sequence} We now give a list of steps to produce the sequence of curves starting with an initial step that gives a modified version of the example from Figure \ref{fig:cross3}. 

{\bf Step 0:} The first point push is along black, then purple. (We will call the part of $\beta$ pushed here \emph{Tail 1}.) This push is illustrated in the second image in Figure~\ref{fig:crossn}. The second point push is along the orange arc. (We will call the part of $\beta$ pushed here \emph{Tail 2.}) This push is illustrated in the third image of Figure~\ref{fig:crossn}. The third push is along the blue arc. (We will call the part of $\beta$ pushed here \emph{Tail 3}.) This push is illustrated in the fourth image of Figure~\ref{fig:crossn}.

{\bf Step 1:} Push Tail 1 along the orange arc. 

{\bf Step 2:} Push Tail 2 along the pink arc.

{\bf Step 3:} Push Tail 1 along the pink arc.

At this point, all three tails are lined up around the back of the surface, with Tail 3 being the outermost, then Tail 2, then Tail 1. The idea of the next pushes is that the tails will take turns looping around blue and stacking into each other.

We note that, from now on, when we say to ``push along the blue arc'', we actually consider many parallel copies of the blue arc such that consecutive copies share one endpoint so that they could be concatenated to create a longer arc (and similarly, the first blue arc shares its first endpoint with the endpoint of the pink). This is so that the total push for each tail of $\beta$ will be along a simple arc in $\gamma$ that forms no bigons with $\gamma$ in $S$---and therefore lifts to arcs that are in minimal position with $\gamma$ in $\tilde S_\gamma.$

{\bf Step $k$ for $k\geq 4$:} If $k= 1\mod 3$, push Tail 3 along the blue arc. If $k= 2 \mod 3,$ push Tail 2 along the blue arc. If $k=0 \mod 3,$ push Tail 1 along the blue arc.

This sequence of point pushes, along with its impact on crossing numbers, is summarized in Table~\ref{table:table}.

\begin{table}[h]
    \centering
    \begin{tabular}{|c|c|c|c|c|c|c|}
        \hline
        \textbf{Step} & \textbf{Push} & \textbf{Tail 1 intersects} & \textbf{Tail 2 intersects} & \textbf{Tail 3 intersects} & \textbf{$C_\gamma(\beta)$} & \textbf{$C_\beta(\gamma)$} \\
        \hline
        0 & T1 Purple & 0,1 & none & none & 2 & 2 \\
        \hline
        0 & T2 Orange & 0,1 & $-1 $& none & 3 & 2 \\
        \hline
       0 & T3 Blue & 0,1 & $-1$ & $-1$ & 3 & 2 \\
        \hline
        1 & T1 Orange & 0,1,2 & $-1$ & $-1$ & 4 & 3 \\
        \hline
        2 & T2 Pink & 0,1,2 & $-1,-2$ & $-1$ & 5 & 3 \\
        \hline
        3 & T1 Pink & 0,1,2,3 & $-1,-2$ & $-1$ & 6 & 4 \\
        \hline
        4 & T3 Blue & 0,1,2,3 & $-1,-2$ & $-1,-2$ & 6 & 4 \\
        \hline
        5 & T2 Blue & 0,1,2,3 & $-1,-2,-3$ & $-1,-2$ & 7 & 4 \\
        \hline
        6 & T1 Blue & 0,1,2,3,4 & $-1,-2,-3$ & $-1,-2$ & 8 & 5 \\
        \hline
        7 & T3 Blue & 0,1,2,3,4 & $-1,-2,-3$ & $-1,-2,-3$ & 8 & 5 \\
        \hline
        8 & T2 Blue & 0,1,2,3,4 & $-1,-2,-3,-4$ & $-1,-2,-3$ & 9 & 5 \\
        \hline
        9 & T1 Blue & 0,1,2,3,4,5 & $-1,-2,-3,-4$ & $-1,-2,-3$ & 10 & 6 \\
        \hline
    \end{tabular}
    \caption{The point pushes done to $\beta$ to achieve the lower bound of Lemma~\ref{prop:nonsepcurve}. Columns 3, 4, and 5 list the levels of $\gamma$ that the lifts of Tails 1, 2, and 3 in $\hat\beta$ intersect, respectively.}
    \label{table:table}
\end{table}

\pit{Calculating $C_\gamma(\beta)$} By definition, $C_\gamma(\beta)$ is the total number of levels of $\gamma$ that $\beta$ intersects. Since the only parts of $\beta$ that are moved by the point pushes are the tails, $C_\gamma(\beta)$ is the total number of levels the tails intersect. 

Next, each point push in the described sequence is along a simple arc in $S$ that is in minimal position with $\gamma$ (rel endpoints). Therefore, any lift of such an arc is also in minimal position with the lifts of $\gamma.$ In particular, any lift of such an arc that intersects $\hat\gamma$ is in minimal position with the lifts of $\gamma$, and it follows from Lemma~\ref{lemma:hitscompact} that the levels of $\gamma$ it hits are all consecutive (with no repeats). Therefore, each time a tail is pushed along an arc, its lift in $\hat\beta$ intersects precisely one more level. The orientations of $\beta$ and $\gamma$ tell us that Tail 1 intersects an additional positive level when pushed while Tails 2 and 3 intersect the next negative level.

Moreover, it follows from induction that pushing Tails 1 and 2 adds one each to $C_\gamma(\beta)$. On the other hand, pushing Tail 3 does nothing since it repeats the levels Tail 2 intersects. See Table~\ref{table:table} for the base cases and the first several inductive steps.

\pit{Calculating $C_\beta(\gamma)$} We first notice that all arcs that are pushed along have their terminal sides on the right side of $\gamma.$ This implies that only lifts in \emph{positive} levels of $\beta$ will ever intersect $\hat\gamma.$ 

Next, we notice that $C_\beta(\gamma)$ will always be one more than the highest absolute value elevation that any of the tails intersects. Let $n$ be the highest absolute value elevation that any tail intersects; it comes from Tail 1 by induction. This came from a total of $n$ point pushes, creating a total point push along $n$ concatenated arcs. The lift of concatenated arc with its head in $\hat\gamma$ has its tail in level $n$ of $\gamma.$ The point push then took a level $n$ elevation of $\beta$ and pushed it to intersect $\hat\gamma.$ The $+1$ comes from the fact that $\hat\gamma$ always intersects level 0 of $\beta$.

To summarize, $C_\beta(\gamma)$ is increased only when Tail 1 is pushed, and the increase is by 1.

\medskip\noindent Overall, pushing Tail 1 increments both $C_\gamma(\beta)$ and $C_\beta(\gamma),$ pushing Tail 2 increments only $C_\gamma(\beta),$ while pushing Tail 3 does nothing (except make room for the other tails). Since we push all three in a cyclic order, we obtain that all crossing number pairs $(C_\gamma(\beta),C_\beta(\gamma))=(2n-2,n)$ or $(2n-1,n)$ are obtained for $n\geq 2.$
\end{proof}

\p{Crossing number bounds for separating curves} Our next main goal is to prove Proposition~\ref{prop:sepcurvessymnum}, which states that crossing numbers are symmetric for separating curves.

A separating curve behaves differently from a nonseparating curve when lifted to its cyclic cover. Let $[\alpha]$ be an isotopy class of separating curves and $\gamma\in[\alpha].$ For any curve $\beta$ that crosses $\gamma,$ $\beta$ must always come back to $\gamma$ on the same side it left from. For example, if we orient $\gamma$ and a curve $\beta$ crosses from the left to the right of $\gamma,$ in its next intersection with $\gamma,$ $\beta$ must cross from the right to the left. This leads to the observation that, if we orient $\gamma$, then, in the cyclic cover corresponding to $[\alpha],$ consecutive levels of $\gamma$ swap orientations.

The following lemmas use this observation to show that compact elevations are ``symmetric'' (in a loose sense of the word) in the cyclic cover.

\begin{lemma} \label{1imples0}
    If isotopic separating curves $\beta$ and $\gamma$ intersect, then their compact elevations in the cyclic cover must also intersect.
\end{lemma}

\begin{proof}
     We proceed by contradiction. Orient $\beta$ and $\gamma$ compatibly and lift these orientations to the cyclic cover. Suppose without loss of generality that $\hat{\beta}$ is on the right side of $\hat{\gamma}$. The other elevations of $\beta$ must then be to the right of the corresponding parallel elevation of $\gamma$. Since $\beta$ and $\gamma$ intersect, $\hat{\beta}$ must intersect a noncompact elevation of $\gamma$ on the right of $\hat{\gamma}$. But this is impossible since $\hat{\beta}$ would have to cross another elevation of $\beta$, contradicting $\beta$ being simple. 
\end{proof}

\begin{lemma} \label{lemma:symint}
 If $\hat{\beta}$ intersects level $\pm n\neq 0$ of $\gamma$, then it also intersects level $\pm(-n+1)$ of $\gamma$.
\end{lemma}

\begin{proof}
     Suppose without loss of generality that $\hat{\beta}$ intersects level $n$ of $\gamma,$ with $n>0.$ Let $\beta_1$ be a minimal subarc of $\hat{\beta}$ with endpoints on level 0 and level $n.$ Let $x$ be an intersection point of $\beta_1$ with the elevation of $\gamma$ at level 1 separating the given level $n$ elevation from the compact elevation.  See Figure \ref{fig:sep}.

    Let $y\in \hat{\gamma}$ be the unique lift of $p(x)$ in the compact component of $p^{-1}(\gamma)$. Let $\beta_2$ be the subarc of the unique lift of $\beta$ that passes through $y$ such that $p(\beta_1)=p(\beta_2).$ 

    Let $b_1$ be the subarc of $\beta_1$ that has its endpoints in level 0 and level 1 of $p^{-1}(\gamma).$ Since $p(x)=p(y),$ there exists a subarc of $\beta_2,$ which we call $b_2,$ such that $p(b_1)=p(b_2).$ We therefore know that $b_2$ also hits exactly 2 levels. To respect the orientations we assigned to the curves, these levels must be level 0 and level 1.

    Since $\beta_1$ essentially intersects $n+1$ levels (Lemma \ref{lemma:hitscompact}), its translate $\beta_2$ does as well (Lemma \ref{lemma: arcs}). Since $\beta_2$ has an endpoint on level 1, and it essentially intersects level 0 by construction, it has its other endpoint on level $-n+1$ (Lemma \ref{lemma: essential2}).

    Moreover, since $\beta_1$ is a part of the level 0 elevation and has an endpoint on level 0, Lemma~\ref{lemma: arcs} implies that $\beta_2$ must be a subset of a level 1 elevation of $\beta$ since the corresponding endpoint is in level 1 of $\gamma.$ Let $\overline{\beta}$ be the elevation of $\beta$ that contains $\beta_2.$ Since $\hat{\beta}$ is separating, we have that $\overline{\beta}$ must be to one side of $\hat{\beta}.$ In particular, it must be to the right (to match the level 1 elevations of $\gamma$ being to the right of $\hat{\gamma}$). Therefore, $\hat{\beta}$ must intersect level $-n+1$ so that $\overline{\beta}$ can stay to the right of $\hat{\beta}.$
\end{proof}

\begin{figure}[h]
\begin{center}
  \includegraphics[width=0.35\textwidth]{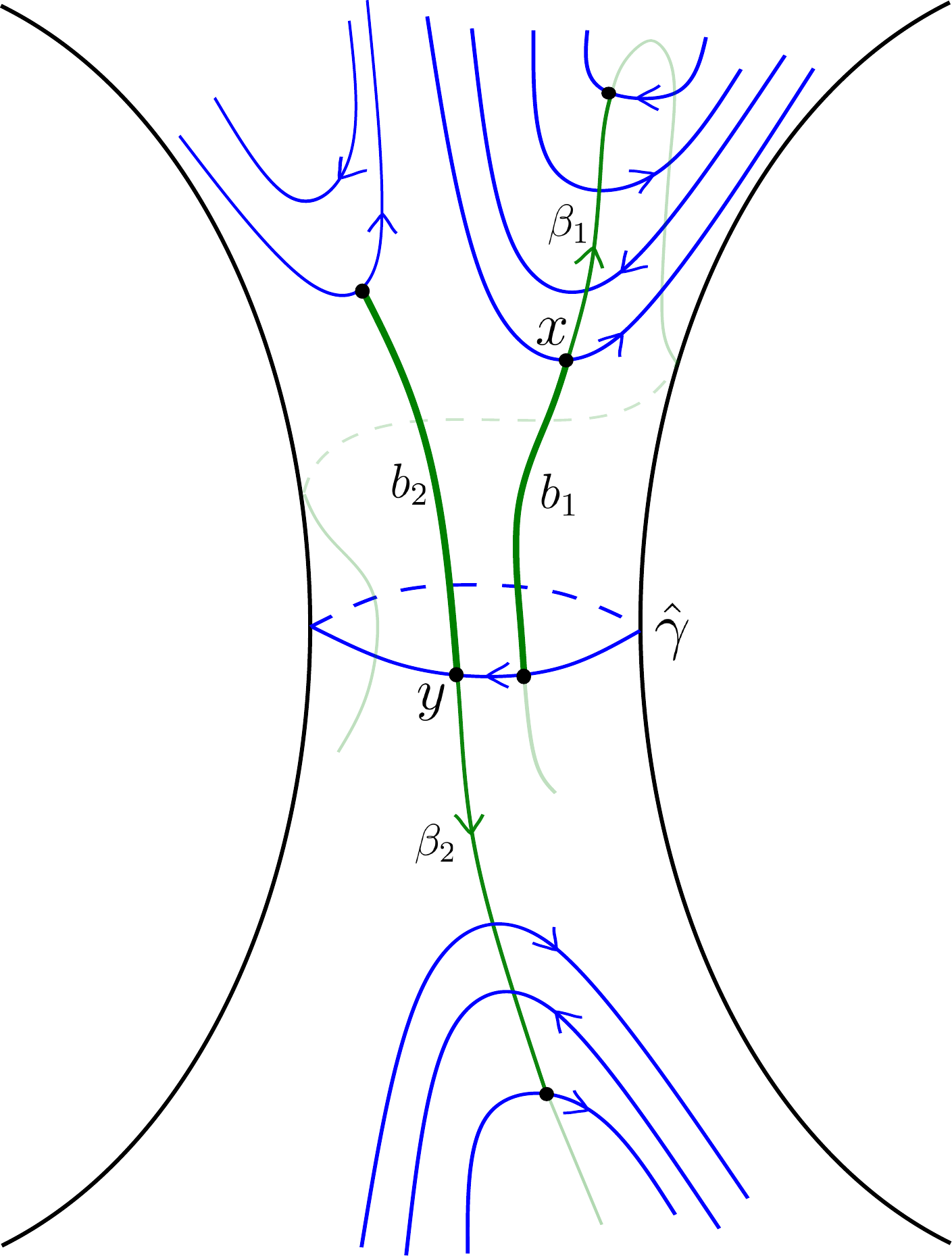}
\caption{The arcs from the proof of Lemma \ref{lemma:symint}}
\label{fig:sep}
\end{center}
\end{figure}

   \begin{proposition} \label{prop:sepcurvessymnum}
       Let $S=S_g$ with $g\geq 2.$ For isotopic separating curves $\beta$,$\gamma$,
       $$C_{\gamma}(\beta) = C_{\beta}(\gamma)$$
       
   \end{proposition}

   \begin{proof}
    Suppose $\hat{\beta}$ intersects level $n$ of $\gamma.$ A similar argument to Lemma~\ref{prop:nonsepcurve} leads us to conclude that $\hat{\gamma}$ intersects level $n$ or level $-n$ of $\beta.$ Due to the orientations of $\gamma$ and $\beta$, if $n$ is odd, $\hat{\gamma}$ intersects level $n$. If $n$ is even, $\hat{\gamma}$ intersects level $-n.$

    By Lemma~\ref{lemma:symint}, the outermost levels of $\gamma$ that $\hat\beta$ intersects differ by at most 1. Moreover, by the same argument used in the previous paragraph, each level of $\gamma$ intersected by $\hat\beta$ leads to a level of $\beta$ that $\hat\gamma$ intersects (and vice versa), completing the equality.
    \end{proof}

\section{Distance bounds}\label{sec:six} Now we are ready to prove the distance results for the single-isotopy-class fine curve graph, starting with the separating case. We first will show a lemma that applies to both separating and nonseparating curves and will eventually lead to the lower bound for both.

\begin{lemma} \label{lowerbound}
    Let $\beta = \beta_0, \beta_1, \ldots, \beta_N$ be a sequence of homotopic essential simple closed curves such that $\beta_n$ is disjoint from $\beta_{n+1}$ for all $n$. Let $L_n$ be the list of levels of $\beta$ that $\hat\beta_n$ intersects, and let $M_n$ be the maximum of the absolute values of the elements of $L_n$ or 0 when $L_n$ is empty. Then $M_n \leq n-1$ for $n > 0$.
\end{lemma}

\begin{proof}
   We proceed by induction. Since $\beta_1$ and $\beta$ are disjoint, $M_1 = 0$ by definition. 
       
    A more interesting base case is that of $M_2.$ We aim to show that $M_2\leq 1.$ Observe that in the cyclic cover, the compact elevation of $\beta_1$ bounds an annulus with the compact elevation of $\beta$, and the other elevations of $\beta_1$ each bound a disk with two boundary points removed with a paired elevation of $\beta$. The elevations of $\beta_1$ can be either on the right or left side or the paired elevation of $\beta$. Since the compact elevation of $\beta_{2}$ is disjoint from the elevations of $\beta_1$, it cannot hit a level $\pm2$ elevation of $\beta$ as it would have to pass through a level $\pm1$ elevation of $\beta$ and thus a $\pm 1$ elevation of $\beta_1$. It follows that $M_{2}\leq 1.$

    For the inductive step, suppose we have $M_n \leq n-1$ for some $n \geq 2$. Then in the cyclic cover, the compact elevation of $\beta_{n+1}$ is between the compact elevation of $\beta_n$ and (without loss of generality) the level 1 elevation of $\beta_n$. (See Figure~\ref{fig:sepcurveboundsNEW} for a schematic.) We claim that the level 1 elevations of $\beta_n$ can only go out as far as the level $M_n +1$ elevations of $\beta$, from which it follows the compact elevation of $\beta_{n+1}$ can only go out as far as the level $M_n + 1$ elevations as well. Otherwise, it would cross an elevation of $\beta_n$. We are done if we show this claim, since then $M_{n+1} \leq M_n + 1 \leq n=(n+1)-1$. 

    Suppose for the sake of contradiction that a level 1 elevation of $\beta_n$ intersects a level $M_n+2$ elevation of $\beta$. (Note that we cannot have a level 1 elevation of $\beta_n$ intersecting a level $-(M_n+2)$ elevation of $\beta$ since $\hat\beta_n$ can only go down to level $-M_n$ of $\beta$.) Therefore, there is a subarc $\zeta_1$ of the level 1 elevation of $\beta_n$ with an endpoint on the $M_n+2$ level of $\beta$. Since $\zeta_1$ is a subset of a level 1 elevation of $\beta_n$ which itself shares endpoints in $\partial \tilde S_\gamma$ with a level 1 elevation of $\beta,$ we have that to intersect a level $M_n+2$ elevation of $\beta,$ $\zeta_1$ must intersect a level 2 elevation of $\beta.$ Since each elevation of $\beta$ is separating and to reach a higher level elevation, you must first cross lower level elevations, $\zeta_1$ must essentially intersect all levels of $\beta$ between 2 and $M_n+2,$ inclusive. Therefore $\zeta_1$ intersects at least $M_n+1$ elevations of $\beta.$

       \begin{figure}[htbp]
\begin{center}
\begin{tikzpicture}
    \node[anchor = south west, inner sep = 0] at (0,0) {\includegraphics[width=0.55\textwidth]{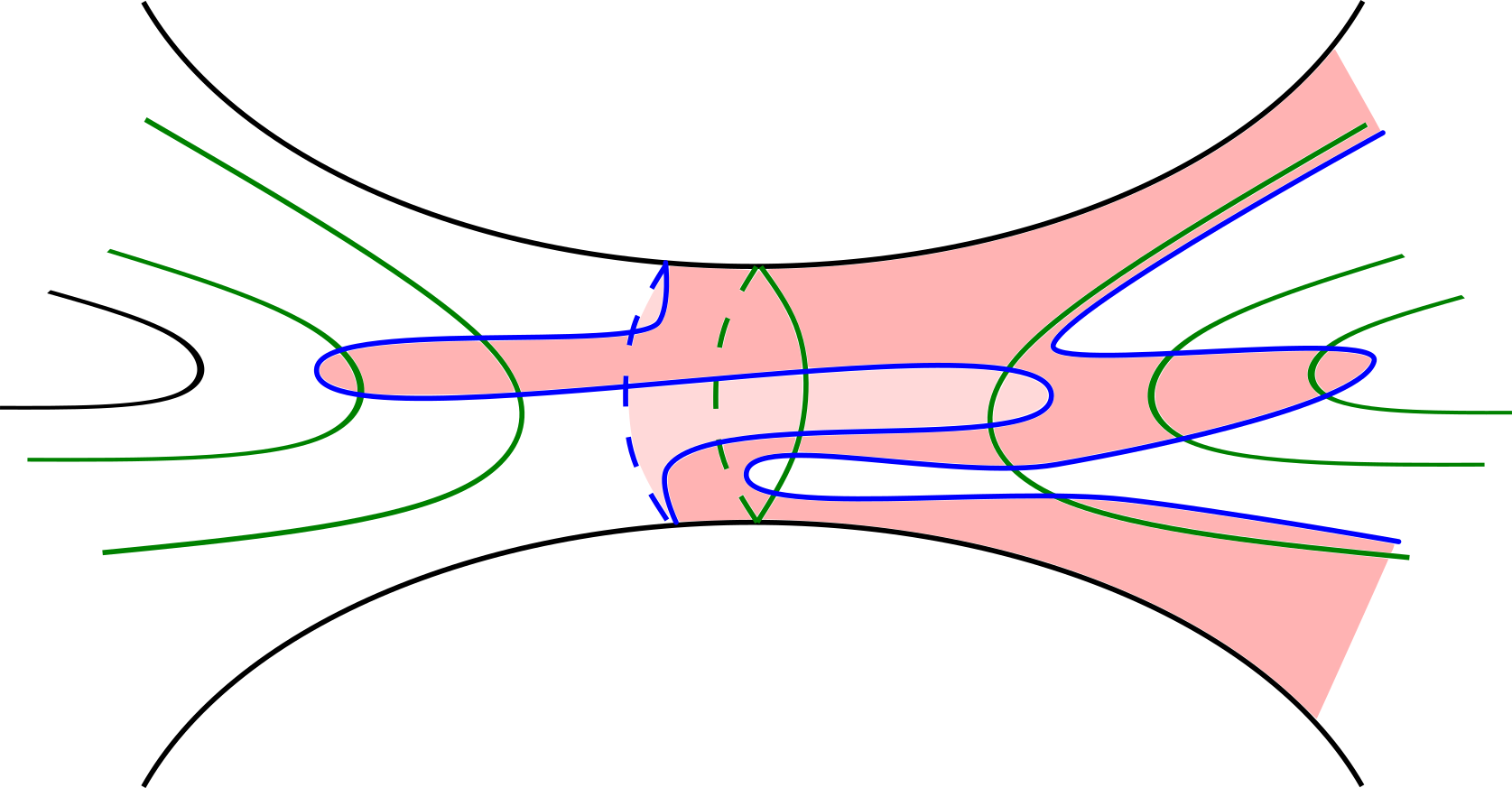}};

    \end{tikzpicture}
\caption{Pictured are lifts of $\beta$ (in green) and $\beta_n$ (in blue) of geodesic path $\beta-\beta_1-\cdots-\beta_n-\beta_{n+1}$. If $\hat\beta_{n+1}$ is to the right of $\hat\beta_n$, then it must be in the region bounded by all level 0 and level 1 elevations of $\beta_n.$ This restricts which levels of $\beta$ the lift $\hat\beta_{n+1}$ can intersect. This idea is used in the proofs of Lemma~\ref{lowerbound}, Proposition~\ref{prop:distboundssep}, and Proposition~\ref{prop:distboundssingleisotopyclass}. 
}\label{fig:sepcurveboundsNEW}
\end{center}
\end{figure}

    We have two cases: either corresponding level 1 elevations of $\beta$ and $\beta_n$ intersect or they do not. If the elevations do intersect, take $\zeta_1$ to have its other endpoint in the corresponding level 1 elevation, define this point of intersection to be $\zeta_2,$ and let $\zeta=\zeta_1\cup \zeta_2=\zeta_1.$ Otherwise, if the corresponding level 1 elevations do \emph{not} intersect, let $\zeta_2$ be an arc connecting the corresponding level 1 elevations of $\beta$ and $\beta_n$ that is disjoint from the lifts of $\beta\cup\beta_n$ in its interior. Let $\zeta=\zeta_1\cup \zeta_2.$

    We now have an arc $\zeta$ from a level 1 elevation of $\beta$ to a level $M_n+2$ elevation of $\beta$ that is almost entirely comprised of a level 1 elevation of $\beta_n.$ 

    Let $x$ be the endpoint of $\zeta$ contained in $\zeta_2.$ Let $x'$ be the translate of $x$ in $\hat\beta$ and let $\zeta'$ be the translate of $\zeta$ with $x'$ as one endpoint. We define $\zeta'_1\subseteq \zeta'$ and $\zeta'_2\subseteq \zeta$ similarly as translates of $\zeta_1$ and $\zeta_2,$ respectively. By Lemma~\ref{lemma: arcs}, $\zeta'_2$ connects corresponding elevations of $\beta$ and $\beta_n$---that is, $\zeta'_2$ connects the compact elevations $\hat\beta$ and $\hat\beta_n.$ It follows that $\zeta'_1\subset \hat\beta_n$. 
    
    Since taking translates of arcs preserves essential intersection number, $\zeta'$ essentially intersects $M_n+2$ elevations of $\beta.$ By Lemma~\ref{lemma:hitscompact}, the list of levels $\zeta'$ intersects when homotoped into minimal position is monotonic, so $\zeta'$ intersects level $\pm(M_n+1).$ It follows that the second endpoint of $\zeta'$---the one contained in $\zeta'_1$ and $\hat\beta_n$---is in a level $\pm(M_n+1)$ elevation of $\beta.$ This contradicts the definition of $M_n,$ which is the largest absolute value of any level of $\beta$ that $\beta_n$ intersects.
    
    We conclude that no level 1 elevation of $\beta_n$ can intersect a $\pm(M_n+2)$ elevation of $\beta.$ Therefore, we have that $M_{n+1}\leq M_n+1 \leq n,$ as desired.
    \end{proof}

\begin{proposition}\label{prop:distboundssep}
    Let $S$ be a closed surface of genus at least two. Let $\beta$ and $\gamma$ be separating curves isotopic to a curve $\alpha$. Let $C = C_{\beta}(\gamma) = C_{\gamma}(\beta)$. Then, 
    \[\dfiber(\beta,\gamma) = \lceil \frac{C}{2} \rceil+1.\]
\end{proposition}

\begin{proof}
    We divide our proof into two parts corresponding to the two inequalities.

    \pit{Upper bound: $\dfiber(\beta,\gamma) \leq \lceil \frac{C}{2} \rceil+1$} Our proof of the upper bound is informed by that of Lemma 4.5 of Bowden--Hensel--Mann--Militon--Webb in the case of tori \cite{BHMMW}. However, in our case, we are able to make some stronger statements and tighten the bound by dividing it nearly in half. 

    The main idea of the proof is to create a path from $\beta$ to $\gamma$ by surgering $\beta$ sequentially. Consider the outermost levels of $\beta$ that $\hat\gamma$ in the cyclic cover. This collection of intersections must form bigons in the cyclic cover; our goal is to remove these bigons, one (or two) level(s) at a time.

    Suppose $\hat \gamma$ hits an even number of levels of $\beta.$ This means that the outermost levels of $\beta$ that $\hat\gamma$ intersects are oriented in opposite directions, so the ends of the annulus are both to the right or both to the left of the outermost intersected levels.
    We can homotope $\beta$ so the lifts of $\beta$ avoid the intersections of $\hat\gamma$ with all of the outermost levels of $\beta.$ The main differences from the case of Bowden--Hensel--Webb are that (1) there are two outermost levels, and (2) $\hat\gamma$ may intersect multiple elevations of the outermost level of $\beta.$ 

     \begin{figure}[h]
\begin{center}
\begin{tikzpicture}
    \node[anchor = south west, inner sep = 0] at (0,0) {\includegraphics[width=0.8\textwidth]{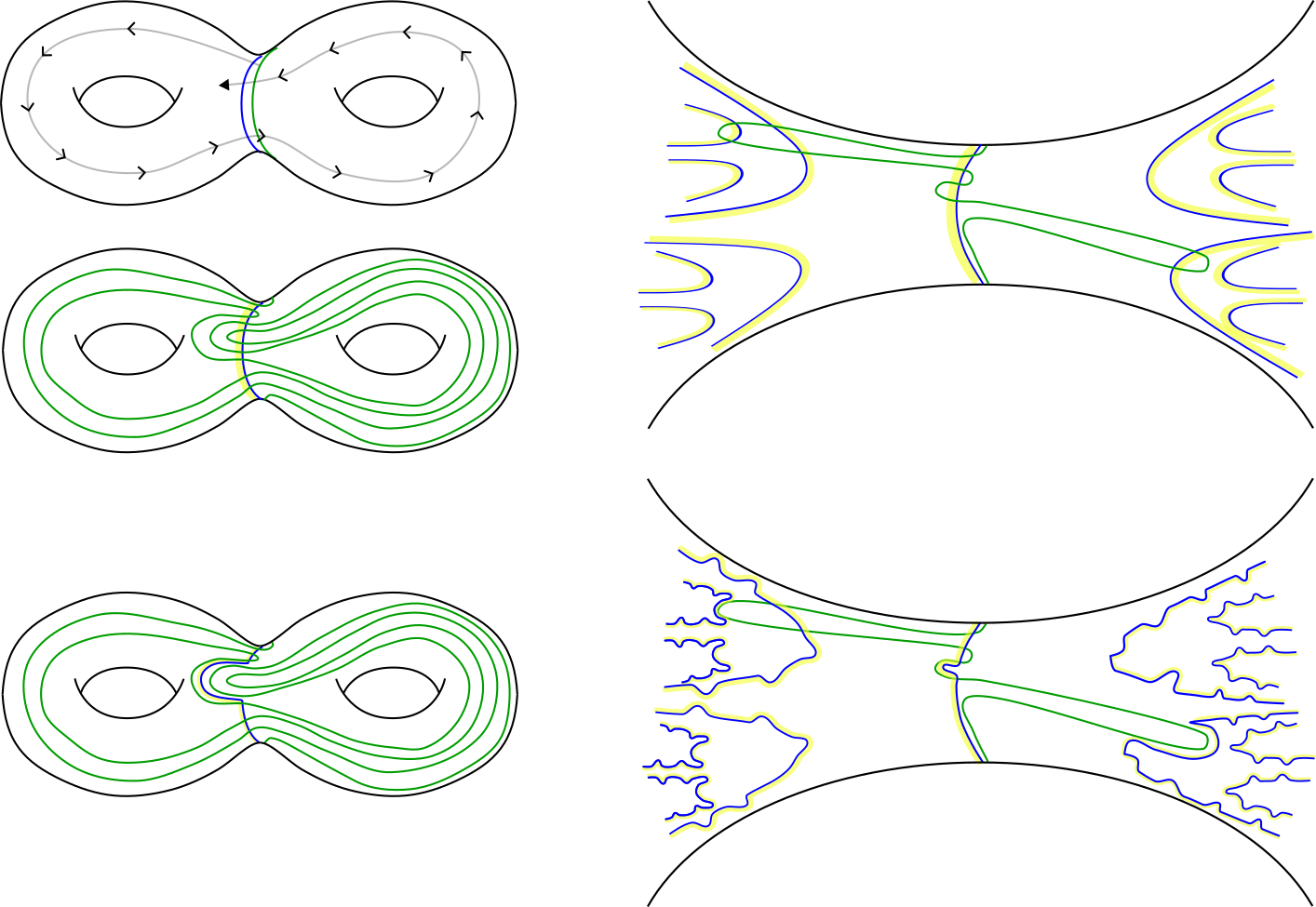}};
    \end{tikzpicture}
\caption{A schematic of the surgery procedure giving the upper bound in Proposition~\ref{prop:distboundssep}. Top left: To obtain a sample $\beta_n$ (green curve), we point push a curve parallel to the $\beta$ (blue). Middle left: $\beta$ and $\beta_n,$ with the left side of $\beta$ highlighted for orientation purposes. Top right: the lift of Middle left to $\Tilde{S}_\gamma.$ Bottom right: homotoping $\beta$ in the cyclic cover to avoid outermost intersections of $\beta_n$ and $\beta.$ Bottom left: projection of Bottom right to $S.$}\label{fig:sepcurvebounds}
\end{center}
\end{figure}

    To resolve these difficulties, we notice that the interiors of bigons formed with the outermost levels in the cyclic cover project to embedded disks in $S.$ Thus, these bigons in the cyclic cover project to either outermost bigons in $S$ or not. Moreover, we notice that outermost bigons have disjoint interiors in $S.$ Therefore, we may homotope $\beta$ to avoid these outermost bigons. This corresponds to pushing $\beta$ ``outward'' in the cyclic cover. We therefore note that the ``pushing'' operation of Bowden--Hensel--Mann--Militon--Webb still applies in our case. A schematic of this procedure can be seen in Figure~\ref{fig:sepcurvebounds}. In particular, we are constructing $\frac{C}{2}$ vertices in our path from $\beta$ to $\gamma,$ resulting in a path of length $\frac{C}{2}+1.$

    Suppose $\hat\gamma$ hits an odd number of levels of $\beta.$ Since $\beta$ and $\gamma$ are separating, Lemma~\ref{lemma:symint} tells us this only happens if these are levels $-n$ through $n.$ We then homotope $\beta$ to remove bigons with exactly one of the outermost levels, leaving us in the previous case of an even number of levels intersected. We point out that the extra vertex along the path is accounted for by the ceiling function.

    \pit{Lower bound: $\dfiber(\beta,\gamma) \geq \lceil \frac{C}{2} \rceil+1$} Suppose $d_{\fine_\alpha}(\beta,\gamma)=N,$ so we have a geodesic of length $N$ from $\beta$ to $\gamma$ given by $\beta = \beta_0, \beta_1, \ldots, \beta_N = \gamma$. Let $L_n$ be the set of levels of $\beta$ that $\beta_n$ intersects, and let $M_n$ be the maximum of the absolute values of the elements of $L_n$. By Lemma \ref{lowerbound}, we have that $N \geq M_N+1$. Additionally, since $\hat \beta_{n+1}$ is on either the right or left side of $\hat\beta_n$, we have that on one side (either positive or negative, whichever $\hat\beta_n$ is on), the maximum absolute value of the levels $\hat\beta_{n+1}$ intersects cannot exceed the maximum absolute value of the levels $\hat\beta_n$ intersects.

    Lemma~\ref{lemma:symint} tells us that either $C=2M_N$ (if $M_N$ is achieved by exactly one of the positive or negative levels) or $C=2M_N+1$ (if $M_N$ is achieved by both the positive and negative levels). In the former case, we have \[C=2M_N \leq 2(N-1),\] so we obtain \[N\geq \frac{C}{2}+1=\lceil\frac{C}{2}\rceil+1.\] In the latter case, we notice that since either the positive or negative side of $\hat\beta_{N}$ is bounded by $\hat\beta_{N-1},$ we have that $M_{N-1}\geq M_N.$ (See Figure~\ref{fig:sepcurveboundsNEW} for a schematic.) Moreover, from Lemma~\ref{lowerbound}, we have that $M_{N-1}\leq N-1-1=N-2.$ Therefore, we have $M_N\leq M_{N-1}\leq N-2$, so \[C=2M_N+1\leq 2(N-2)+1=2N-3.\] We rearrange this to conclude that \[N\geq \frac{C+3}{2}\geq \lceil\frac{C}{2}\rceil+1,\] as desired.
\end{proof}

\begin{proposition}\label{prop:distboundssingleisotopyclass}
    Let $S$ be a closed surface of genus at least two. Let $\beta$ and $\gamma$ be nonseparating curves isotopic to a curve $\alpha$. Then, 
    \[ \lceil \frac{\max\{C_{\gamma}(\beta), C_{\beta}(\gamma)\}}{2} \rceil + 1 \leq \dfiber(\beta,\gamma)\leq \min\{C_{\gamma}(\beta),C_{\beta}(\gamma)\}+1.\]
\end{proposition}

\begin{proof}
    We split the proof into the upper bound and the lower bound.
    
    \pit{Upper bound: $\dfiber(\beta,\gamma)\leq \min\{C_{\gamma}(\beta),C_{\beta}(\gamma)\}+1$} The proof of the upper bound is very similar to that of Proposition~\ref{prop:distboundssep}. Orient $\gamma$ and consider the oriented lifts to the cyclic cover. Consider the levels of $\gamma$ intersected by $\hat\beta.$ (Note that, unlike the separating case, these may not be symmetric around 0.) Without loss of generality, the elevations at the outermost levels may have orientations such that all right hand sides of outermost intersected levels point toward the ends of the cylinder, or only some of them do.

    If all right hand sides of outermost intersected levels point toward the ends (or all point away from the ends), we homotope $\beta$ to avoid all bigons formed by $\hat\beta$ and the outermost levels of $\gamma$ that $\hat\beta$ intersects, as in the proof of Proposition~\ref{prop:distboundssep} (though with the roles of $\beta$ and $\gamma$ swapped).

    If there is a mix of directions, two homotopies should occur: one for each direction (also as in the proof of Proposition~\ref{prop:distboundssep}). Since this resolves intersections with \emph{both} outermost levels of $\gamma$ that $\hat\beta$ intersects, we have decreased $C_\gamma(\beta)$ by two and added at most two vertices to our path.

    Performing this operation repeatedly leads to the construction of $C_\gamma(\beta)$ curves in a path between $\beta$ and $\gamma,$ resulting in a path of length $C_\gamma(\beta)+1.$ Since we can interchange $\beta$ and $\gamma$ throughout the proof, we have that $\dfiber(\beta,\gamma)\leq\min\{C_\gamma(\beta),C_\beta(\gamma)\}+1,$ as desired.

    \pit{Lower bound: $\lceil \frac{\max\{C_{\gamma}(\beta), C_{\beta}(\gamma)\}}{2} \rceil + 1 \leq \dfiber(\beta,\gamma)$} The proof of the lower bound is similar to that of the separating case. Suppose we have a geodesic of length $N$ from $\beta$ to $\gamma$ given by $\beta = \beta_0, \beta_1, \ldots, \beta_N = \gamma$. Let $L_n$ be the set of levels of $\beta$ that $\beta_n$ intersects, and let $M_n$ be the maximum of the absolute values of the elements of $L_n$. By Lemma \ref{lowerbound}, we have that $N \geq M_N+1$. 
    
     When the difference between the number of positive levels of $L_N$ with the number of negative levels is at most one (so $C_\beta(\gamma)=2M_N$ or $C_\beta(\gamma)=2M_N+1$), then the bounds in the proof of the separating case apply. 

     Otherwise, $C_\beta(\gamma)<2M_N\leq 2(N-1),$ so $N>\frac{C_\beta(\gamma)}{2}+1.$ We conclude that $N\geq \lceil\frac{C_\beta(\gamma)}{2}\rceil+1.$

    Since this proof works with roles of $\beta$ and $\gamma$ reversed, we choose the maximum of $C_{\beta}(\gamma)$ and $C_{\gamma}(\beta)$ to bound $N$ from below. 
\end{proof}

With the above in mind, we are ready to construct a sequence of isotopic curves with distance from the first curve increasing by one.

\begin{corollary} \label{cor:exactdistance}
For any surface with an essential curve $\alpha$, we have an explicit construction of a sequence $\{\alpha_n\}_{n=0}^\infty$ of curves isotopic to $\alpha$ such that $\dfiber(\alpha_0, \alpha_n) = n$.
\end{corollary}

\begin{proof}
    We split our proof into two cases: $\alpha$ is nonseparating or $\alpha$ is separating.

    \pit{Nonseparating} Let $\beta$ and $\gamma$ be isotopic nonseparating curves and define $C=\max\{C_\beta(\gamma),C_\gamma(\beta)\}$. We know from Lemma~\ref{prop:nonsepcurve} that $\lfloor \frac{\max\{C_{\gamma}(\beta), C_{\beta}(\gamma)\}}{2}\rfloor +1   \leq \min\{C_{\gamma}(\beta), C_{\beta}(\gamma)\}$ and from (the proof of) Lemma~\ref{lemma:achievedminimum} that we can achieve this bound by point-pushing $\beta$ to create $\gamma.$ In fact, if $C$ is odd, then Proposition~\ref{prop:distboundssingleisotopyclass} gives us that \[\min\{C_\gamma(\beta),C_\beta(\gamma)\}+1=(\lfloor\frac{C}{2}\rfloor +1)+1=\lceil\frac{C}{2}\rceil+1 \leq \dfiber(\beta,\gamma)\leq \min\{C_\gamma(\beta),C_\beta(\gamma)\}+1.\] We conclude that the constructed curves would have distance $\min\{C_\gamma(\beta),C_\beta(\gamma)\}+1$. We can choose $\beta$ and $\gamma$ to make this quantity equal to $n.$

    We then consider a homeomorphism $\varphi$ of the surface taking $\beta$ to $\alpha=\alpha_0$ and set $\alpha_n=\varphi(\gamma).$ 

    \pit{Separating} The separating case is much more straightforward. Let $\gamma$ be a curve isotopic to $\alpha$; this will be $\alpha_0$. We will now construct $\alpha_n.$ By Proposition~\ref{prop:distboundssep}, it suffices to point push $\gamma$ to create a curve $\alpha_n$ such that $C_\gamma({\alpha_n})=2(n-1)$ or $C_\gamma({\alpha_n})=2n-3.$ This can be accomplished by point-pushing $\gamma$ repeatedly around a curve that intersects $\gamma$ twice transversely (in minimal position). 
    \end{proof}

\pit{The torus case} For completeness, we mention the corresponding result for the single-isotopy-class fine curve graph of the torus. The cyclic cover is much simpler in this situation, and the crossing number as we have defined it still applies. The upper bound on distance between two isotopic curves is Lemma 4.5 of Bowden–Hensel–Mann–Militon–Webb \cite{BHMMW}. We note that the existing result is only an upper bound to account for defining the edges of the fine curve graph of a torus to connect curves that are disjoint or intersect exactly once. However, we take the fine curve graph of the torus to have \emph{only edges that correspond to disjointness.}

Let $\alpha$ be an essential simple closed curve in the torus $T$ and let $\tilde S_\alpha$ be the annular (cyclic) cover corresponding to $[\alpha]\in\pi_1(T).$ Define crossing number as above.

We briefly mention the symmetry of crossing number for the torus.

\begin{lemma}\label{lemma:torussym}
    Let $\beta,\gamma\in[\alpha]$ be curves in the torus. Then $C_\beta(\gamma)=C_\gamma(\beta).$
\end{lemma}

\begin{proof}
    Suppose $C_\gamma(\beta)=n.$ We will calculate $C_\beta(\gamma)$ by counting how many lifts of $\beta$ in $\tilde S_\alpha$ a chosen lift of $\gamma$ intersects. Let $\hat\gamma$ be this chosen lift.

    We can number the lifts of $\gamma$ by integers such that any two lifts that cobound an annulus have labels differing by 1. (By orienting $\gamma$, and therefore orienting its lifts, we can increase the label by 1 each time we consider the next lift to the right.) Let $\hat\gamma=\gamma_n.$ 

    We can similarly number the lifts of $\beta$, and we let $\beta_1$ be the lift that intersects $\gamma_1,\ldots,\gamma_n.$ Then $\beta_2$ intersects $\gamma_2,\ldots,\gamma_{n+1}$ and $\beta_0$ intersects $\gamma_0,\ldots,\gamma_{n-1}.$ Overall, since the lifts of $\gamma$ intersected $\beta_i$ and $\beta_{i+1}$ differ by adding one to each subscript, we obtain that precisely $\beta_1,\ldots,\beta_n$ intersect $\gamma_n=\gamma,$ as desired.
\end{proof}

To reflect this symmetry, we denote the crossing number between $\beta$ and $\gamma$ by $C(\beta,\gamma).$

\begin{proposition}\label{prop:disttorus}
    Let $\beta$ and $\gamma$ be curves isotopic to a curve $\alpha$ in the torus. Then 
    \[ \dfiber(\beta,\gamma)= C(\beta,\gamma)+1.\]
\end{proposition}

\begin{proof}
    Since $\dfiber(\beta,\gamma)\geq C(\beta,\gamma)+1$ was proven in Lemma 4.5 of \cite{BHMMW}, it remains to show that $\dfiber(\beta,\gamma)\leq C(\beta,\gamma)+1.$
    
    We prove this result by induction on $\dfiber(\beta,\gamma).$ In the base case, $\dfiber(\beta,\gamma)=1,$ so $\beta$ and $\gamma$ are disjoint and therefore $C=0,$ satisfying the inequality.

    As an inductive hypothesis, suppose that for any pair of curves $\beta$ and $\gamma$ with $\dfiber(\beta,\gamma)= n-1,$ we have $n-1=\dfiber(\beta,\gamma)\geq C(\beta,\gamma).$

    Suppose now that $\dfiber(\beta,\gamma)=n,$ so we can find a geodesic path $\beta=\beta^0-\beta^1-\cdots-\beta^n=\gamma.$ Let $\hat\gamma$ be a lift of $\gamma$; we aim to count the number of elevations of $\beta$ it intersects. We know that $\hat\gamma$ is between two elevations of $\beta^{n-1}$---call these $\beta^{n-1}_1$ and $\beta^{n-1}_2$---and can only intersect the elevations of $\beta$ between (without loss of generality) the minimum elevation $\beta^{n-1}_1$ intersects and the maximum elevation $\beta^{n-1}_2$ intersects. Since the minimum and maximum elevations differ by at most $(n-1)-1+1=n-1,$ we have that $\gamma$ can intersect at most $n-1$ levels of $\beta.$ We conclude that $\dfiber(\beta,\gamma)\leq C(\beta,\gamma)+1,$ as desired.
\end{proof}

\bibliographystyle{plainurl} 
\bibliography{mainbib}

\end{document}